\documentclass[10pt,letterpaper]{amsart}
\usepackage[margin=1in]{geometry}

\usepackage{amsmath,amssymb,amsthm,mathrsfs,bbm,comment,units,enumitem,xcolor,adforn,pgfornament,fancyhdr,lastpage,url}
\usepackage[colorlinks, linktocpage, breaklinks]{hyperref}

\usetikzlibrary{arrows}

\renewcommand{\leq}{\leqslant}
\renewcommand{\geq}{\geqslant}

\newcommand{\concat}{{^\frown}}

\newtheoremstyle{mythm}
{.5\baselineskip}	
{.5\baselineskip}	
{}		
{}		
{\bf}	
{. }		
{ }		
{}		

\theoremstyle{mythm}
\newtheorem{theorem}{Theorem}[section]	
\newtheorem{lemma}[theorem]{Lemma}
\newtheorem{proposition}[theorem]{Proposition}
\newtheorem{corollary}[theorem]{Corollary}
\newtheorem{definition}[theorem]{Definition}
\newtheorem{example}[theorem]{Example}

\newtheorem{remark}[theorem]{Remark}
\newtheorem{question}{Question}

\title{On Traditional Menger and Rothberger Variations}

\author{Christopher Caruvana}
\address{School of Sciences, Indiana University Kokomo, 2300 S. Washington Street, Kokomo, IN 46902 USA}
\email{caruvana@gmail.com}
\urladdr{https://chcaru.pages.iu.edu/}

\author{Steven Clontz}
\address{College of Arts and Sciences, University of South Alabama, 411 N University Blvd North, Mobile, AL 36688, USA}
\email{sclontz@southalabama.edu}
\urladdr{https://clontz.org}

\author{Jared Holshouser}
\address{Department of Mathematics, Norwich University, 158 Harmon Drive, Northfield, VT 05663 USA}
\email{JHolshou@norwich.edu}
\urladdr{https://jaredholshouser.github.io/}

\date{\today}


\begin{document}

\maketitle

\begin{abstract}
    We present a comprehensive report on the relationships between variations of the Menger and Rothberger selection properties
    with respect to \(\omega\)-covers and \(k\)-covers in the most general topological setting and
    address the finite productivity of some of these properties.
    We collect various examples that separate certain properties and
    we carefully identify which separation axioms simplify aspects of these properties.
    We finish with a consolidated list of open questions focused on topological examples.
\end{abstract}

\section{Introduction}

The purpose of this paper is to bear the relationships, summarized in Figures \ref{fig:KOmega} and \ref{fig:Omega},
between standard variations of the Menger and Rothberger properties in the most general topological settings
and to identify unresolved questions.
En route to proving the implications of the above-mentioned figures, we also prove novel results
about the productivity of certain properties related to selection principles
using \(k\)-covers (Theorems \ref{StratKRothProductive}, \ref{thm:StraKMengerProductive}, and \ref{MarkovKProductive}).
The implications of Figure \ref{fig:KOmega} are mostly established with Theorem \ref{thm:KImpliesOmegaImpliesOpen};
the implications of Figure \ref{fig:Omega} have been established previously, but in some cases under certain
assumptions about separation axioms.
Our contribution to this line is to remove those assumptions and still obtain the implications.

As an example, relationships between some variants of the Menger and Rothberger properties have originally surfaced in the context of \(C_p\)-theory,
where one usually assumes every space considered is Tychonoff.
One may wish to apply such equivalences to, for example, Hausdorff spaces, and consequently wonder if the
Tychonoff assumption is necessary.
Such equivalances, as well as others, will be proved herein without assuming any separation axioms.

We note that this work was inspired by updating the \(\pi\)-base \cite{PiBase}, a database that started as a digital expansion of \cite{Counterexamples}, and identifying certain gaps.
As such, some examples presented here will bear the names used in \cite{Counterexamples} and their ID in the \(\pi\)-base,
where applicable.

The paper is structured as follows.
Sections \ref{section:CoverPerspectives} and \ref{section:Background} review the notions we'll be working with
in this manuscript.
Section \ref{sec:GeneralResults} contains most of the theory which culminates in Theorem \ref{thm:KImpliesOmegaImpliesOpen},
though many results are presented for their own interest.
Section \ref{sec:GeneralResults} is developed through various subsections:
Section \ref{subsection:FinitePowers} covers the relationships between the \(\omega\)-variants and finite
powers of a space, Section \ref{subsection:CompactCovers} discusses the \(k\)-cover analogues, and
Section \ref{subsection:RelatingTheGames} summarizes some basic conclusions based on the previous two sections.
Sections \ref{subsection:FinitePowers} and \ref{subsection:CompactCovers} are organized in terms of increasing strength,
in a sense, and that order is motivated by the fact that traditional selection principles are the most
widely studied of the levels included here and so should appear sooner rather than later.
We also typically lead with the finite-selection versions before single-selection versions,
a choice also motivated by \emph{strength}.
Once the implications and equivalences have been proved, Section \ref{section:Examples} provides
a list of examples that show that certain implications do not reverse.
Lastly, Section \ref{section:Questions} collects some questions the authors have not been able to answer,
and is mostly a search for examples of topological spaces that separate what seem to be very similar
properties.

Unless otherwise noted, no separation axioms are assumed, so when we say ``for any space \(X\),'' there is no implicit
assumption that \(X\) is, for example, Hausdorff.
When we do require separation axioms, we will use the term \emph{regular} without assuming \(T_1\);
we will use \(T_3\) for regular and \(T_1\).
Any terms used without being defined are to be understood as in \cite{Engelking}.

\begin{figure}
    \centering
        \begin{tikzpicture}
            \node (MKR) at (-8,2.5) {Markov \(k\)-Rothberger};
            \node (SKR) at (-3.5,2.5) {Strat. \(k\)-Rothberger};
            \node (KR) at (0,2.5) {\(k\)-Rothberger};
            
            \node (MKM) at (-8,1.25) {Markov \(k\)-Menger};
            \node (SKM) at (-3.5,1.25) {Strat. \(k\)-Menger};
            \node (KM) at (0,1.25) {\(k\)-Menger};
            \node (KL) at (3.5,1.25) {\(k\)-Lindel{\"{o}}f};
    
            \node (MOM) at (-8,0) {Markov \(\omega\)-Menger};
            \node (SOM) at (-3.5,0) {Strat. \(\omega\)-Menger};
            \node (OM) at (0,0) {\(\omega\)-Menger};
            \node (OL) at (3.5,0) {\(\omega\)-Lindel{\"{o}}f};
            
            \node (MOR) at (-8,-1.25) {Markov \(\omega\)-Rothberger};
            \node (SOR) at (-3.5,-1.25) {Strat. \(\omega\)-Rothberger};
            \node (OR) at (0,-1.25) {\(\omega\)-Rothberger};
    
            \draw [-implies,double equal sign distance] (MKR) -- (SKR) node [midway, above] {\scriptsize \(\not\!\Leftarrow\)};
            \draw [-implies,double equal sign distance] (SKR) -- (KR);
            \draw [-implies,double equal sign distance] (MKR) -- (MKM);
            \draw [-implies,double equal sign distance] (SKR) -- (SKM);
            \draw [-implies,double equal sign distance] (KR) -- (KM);
            \draw [-implies,double equal sign distance] (MKM) -- (SKM) node [midway, above] {\scriptsize \(\not\!\Leftarrow\)};
            \draw [-implies,double equal sign distance] (SKM) -- (KM);
            \draw [-implies,double equal sign distance] (KM) -- (KL) node [midway, above] {\scriptsize \(\not\!\Leftarrow\)};
    
            \draw [-implies,double equal sign distance] (MOM) -- (SOM) node [midway, above] {\scriptsize \(\not\!\Leftarrow\)};
            \draw [-implies,double equal sign distance] (SOM) -- (OM);
            \draw [-implies,double equal sign distance] (OM) -- (OL) node [midway, above] {\scriptsize \(\not\!\Leftarrow\)};
    
            \draw [-implies,double equal sign distance] (KM) -- (OM) node [midway, right] {\scriptsize \(\not\Uparrow\)};
            \draw [-implies,double equal sign distance] (KL) -- (OL);
    
            \draw [-implies,double equal sign distance] (MOR) -- (SOR) node [midway, above] {\scriptsize \(\not\!\Leftarrow\)};
            \draw [-implies,double equal sign distance] (SOR) -- (OR);
    
            \draw [-implies,double equal sign distance] (MOR) -- (MOM) node [midway, right] {\scriptsize \(\not\hspace{0.4pt}\Downarrow\)};
            \draw [-implies,double equal sign distance] (SOR) -- (SOM) node [midway, right] {\scriptsize \(\not\hspace{0.4pt}\Downarrow\)};
            \draw [-implies,double equal sign distance] (OR) -- (OM);
    
            \draw [-implies,double equal sign distance] (SKM) -- (SOM) node [midway, right] {\scriptsize \(\not\Uparrow\)};
            \draw [-implies,double equal sign distance] (MKM) -- (MOM) node [midway, right] {\scriptsize \(\not\Uparrow\)};
        \end{tikzpicture}
    \caption{\(k\)- and \(\omega\)-variants in ZFC}
    \label{fig:KOmega}
\end{figure}
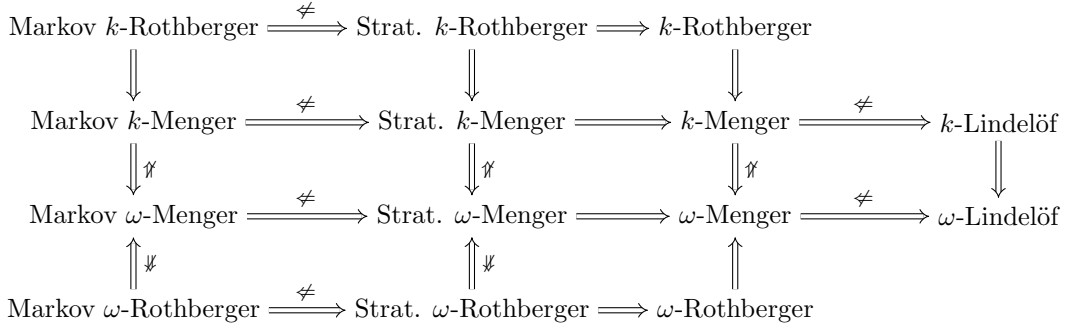

\begin{figure}
    \centering
    \begin{tikzpicture}
        \node (MOM) at (-8,1.25) {Markov \(\omega\)-Menger};
        \node (SOM) at (-3.5,1.25) {Strat. \(\omega\)-Menger};
        \node (OM) at (0,1.25) {\(\omega\)-Menger};
        \node (OL) at (3.5,1.25) {\(\omega\)-Lindel{\"{o}}f};
        
        \node (MOR) at (-8,2.5) {Markov \(\omega\)-Rothberger};
        \node (SOR) at (-3.5,2.5) {Strat. \(\omega\)-Rothberger};
        \node (OR) at (0,2.5) {\(\omega\)-Rothberger};

        \node (MR) at (-8,3.75) {Markov Rothberger};
        \node (SR) at (-3.5,3.75) {Strat. Rothberger};
        \node (R) at (0,3.75) {Rothberger};

        \node (MM) at (-8,0) {Markov Menger};
        \node (SM) at (-3.5,0) {Strat. Menger};
        \node (M) at (0,0) {Menger};
        \node (L) at (3.5,0) {Lindel{\"{o}}f};

        \draw [-implies,double equal sign distance] (MOM) -- (SOM) node [midway, above] {\scriptsize \(\not\!\Leftarrow\)};
        \draw [-implies,double equal sign distance] (SOM) -- (OM);
        \draw [-implies,double equal sign distance] (OM) -- (OL) node [midway, above] {\scriptsize \(\not\!\Leftarrow\)};

        \draw [-implies,double equal sign distance] (MOR) -- (SOR) node [midway, above] {\scriptsize \(\not\!\Leftarrow\)};
        \draw [-implies,double equal sign distance] (SOR) -- (OR);

        \draw [-implies,double equal sign distance] (MOR) -- (MOM) node [midway, right] {\scriptsize \(\not\Uparrow\)};
        \draw [-implies,double equal sign distance] (SOR) -- (SOM) node [midway, right] {\scriptsize \(\not\Uparrow\)};
        \draw [-implies,double equal sign distance] (OR) -- (OM) node [midway, right] {\scriptsize \(\not\Uparrow\)};

        \draw [implies-implies,double equal sign distance] (MM) -- (MOM);
        \draw [implies-implies,double equal sign distance] (SM) -- (SOM);
        \draw [-implies,double equal sign distance] (OM) -- (M);
        \draw [-implies,double equal sign distance] (OL) -- (L) node [midway, right] {\scriptsize \(\not\Uparrow\)};

        \draw [-implies,double equal sign distance] (MM) -- (SM) node [midway, above] {\scriptsize \(\not\!\Leftarrow\)};
        \draw [-implies,double equal sign distance] (SM) -- (M);
        \draw [-implies,double equal sign distance] (M) -- (L) node [midway, above] {\scriptsize \(\not\!\Leftarrow\)};

        \draw [-implies,double equal sign distance] (MR) -- (SR) node [midway, above] {\scriptsize \(\not\!\Leftarrow\)};
        \draw [-implies,double equal sign distance] (SR) -- (R);

        \draw [-implies,double equal sign distance] (OR) -- (R);

        \draw [implies-implies,double equal sign distance] (MOR) -- (MR);
        \draw [implies-implies,double equal sign distance] (SOR) -- (SR);
    \end{tikzpicture}
    \caption{\(\omega\)- and standard-variants in ZFC}
    \label{fig:Omega}
\end{figure}
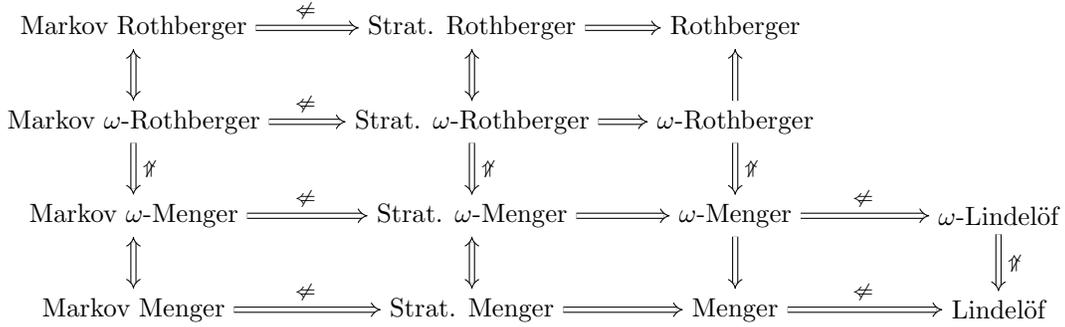

\section{Three perspectives on cover collections} \label{section:CoverPerspectives}

The definition of \(\mathcal O\), the collection of all open covers of a space \(X\),
is standard and used consistently through the literature (with few exceptions, where \(X\) is excluded from open covers; e.g. \cite{OspiovCompactOpen}).
However, the reader should note that there are two standard definitions for \(\Omega\),
one which merely requires that for each \(\mathcal W\in\Omega\) and finite \(F\subseteq X\),
there exists \(W\in\mathcal W\) with \(F\subseteq W\) \cite{GerlitsNagy},
and another which additionally disallows \(X\in\mathcal W\) for each \(\mathcal W\in\Omega\)
\cite{scheepers1997sequential}; i.e., the cover is not ``trivial''. And on occasion, authors
find it necessary to consider covers as sequences rather than sets \cite{CHStar}.

This ambiguity can result in some heartburn for the careful mathematician: do results for one
\(\Omega\) hold for the other? And why is this discrepancy in the literature in the first
place?

To understand this better, we present several characterizations used in the literature
for the standard cover collections \(\mathcal O,\Lambda,\Omega,\Gamma,\mathcal K\).

\begin{definition}
    Let \(\mathcal O_1=\mathcal O_2\) collect all open covers of a topological space \(X\).

    Let \(\mathcal O_3\) collect all transfinite sequences of open sets
    \(\langle U_\beta\rangle_{\beta<\alpha}\) such that
    \(\{U_\beta:\beta<\alpha\}\) forms an open cover of \(X\).
\end{definition}

\begin{definition}
    Let \(\Lambda_1\) collect all open covers \(\mathcal L\) of \(X\) such that
    for each \(x\in X\), \(\{L\in\mathcal L:x\in L\}\) is infinite.
    
    Let \(\Lambda_2\) collect all open covers \(\mathcal L\) of \(X\) such that either,
    for each \(x\in X\), \(\{L\in\mathcal L:x\in L\}\) is infinite,
    or \(X\in\mathcal L\).

    Let \(\Lambda_3\) collect all transfinite sequences of open sets
    \(\langle L_\beta\rangle_{\beta<\alpha}\) such that either
    for each \(x\in X\), \(\{\beta<\alpha:x\in L_\beta\}\) is infinite,
    or \(X=L_\beta\) for some \(\beta<\alpha\).

    These are known as \emph{large} or \(\lambda\)-covers.
\end{definition}

\begin{definition}
    Let \(\Omega_1\) collect all open covers \(\mathcal W\) of \(X\) such that
    for each finite \(F\subseteq X\), there exists \(W\in\mathcal W\) with \(F\subseteq W\),
    and \(X\not\in\mathcal W\).
    
    Let \(\Omega_2\) collect all open covers \(\mathcal W\) of \(X\) such that
    for each finite \(F\subseteq X\), there exists \(W\in\mathcal W\) with \(F\subseteq W\).

    Let \(\Omega_3\) collect all transfinite sequences of open sets
    \(\langle W_\beta\rangle_{\beta<\alpha}\) such that
    for each finite \(F\subseteq X\), there exists \(\beta<\alpha\) with \(F\subseteq W_\beta\).

    These are known as \(\omega\)-covers.
\end{definition}

\begin{definition}
    Let \(\Gamma_1\) collect all open covers \(\mathcal C\) of \(X\) such that
    for each \(x\in X\), \(\{C\in\mathcal C:x\in C\}\) is infinite and co-finite,
    and \(X\not\in\mathcal C\).
    
    Let \(\Gamma_2\) collect all open covers \(\mathcal C\) of \(X\) such that either
    for each \(x\in X\), \(\{C\in\mathcal C:x\in C\}\) is infinite and co-finite,
    or \(X\in\mathcal C\) and \(\mathcal C\) is finite.

    Let \(\Gamma_3\) collect all transfinite sequences of open sets
    \(\langle C_\beta\rangle_{\beta<\alpha}\) such that either
    for each \(x\in X\), \(\{\beta<\alpha:x\in C_\beta\}\) is infinite and co-finite.

    These are known as \(\gamma\)-covers.
\end{definition}

\begin{definition}
    Let \(\mathcal K_1\) collect all open covers \(\mathcal V\) of \(X\) such that
    for each compact \(K\subseteq X\), there exists \(V\in\mathcal V\) with \(K\subseteq V\),
    and \(X\not\in\mathcal V\).
    
    Let \(\mathcal K_2\) collect all open covers \(\mathcal V\) of \(X\) such that
    for each compact \(K\subseteq X\), there exists \(V\in\mathcal V\) with \(K\subseteq V\).

    Let \(\mathcal K_3\) collect all transfinite sequences of open sets
    \(\langle V_\beta\rangle_{\beta<\alpha}\) such that
    for each compact \(K\subseteq X\), there exists \(\beta<\alpha\) with \(K\subseteq V_\beta\).

    These are known as \(k\)-covers.
\end{definition}

We have then the following relationships.

\begin{proposition}
    For each \(i\in\{1,2,3\}\),
    \(\mathcal K_i\subseteq\Omega_i\) and
    \(\Gamma_i\subseteq\Omega_i\subseteq\Lambda_i\subseteq\mathcal O_i\).
\end{proposition}

Perhaps motivating the disqualification of \(X\) from an \(\omega\)-cover
as in \(\Omega_1\) is the guarantee that it disallows all finite \(\omega\)-covers.

\begin{proposition}
    If \(\mathcal W\in\Omega_2\) is finite,
    then \(X\in\mathcal W\) (and thus \(\mathcal W\not\in\Omega_1\)).
\end{proposition}
\begin{proof}
    We prove the contrapositive by assuming \(X\not\in\mathcal W=\{W_0,\dots,W_n\}\).
    Choose \(x_i\in X\setminus W_i\); it follows that
    \(\{x_0,\dots,x_n\}\not\subseteq W_i\) for any \(i\leq n\); therefore
    \(\mathcal W\not\in\Omega_2\).
\end{proof}

It is the preference of the authors of this manuscript
to assume \(i\in\{2,3\}\); the literature is
filled with minor errata that arise when disallowing \(X\) in the open cover.
(E.g., a common technique to obtain an \(\omega\)-cover from an arbitrary open cover
is to close it under finite unions; however, if the open cover contains a finite subcover,
this would not obtain an \(\omega\)-cover in the sense of \(i=1\).) Furthermore,
while it's out of scope to explore in-depth here, the case where \(X\) belongs to a cover
has a nice analog in \(C_p\)-theory: the \(\gamma\)-covers in \(X\) correspond
to the sequences converging to \(\mathbf 0\) in \(C_p(X)\), and when
\(X\) belongs to the \(\gamma\)-cover, we may consider a trivial sequence in \(C_p(X)\).

Regardless,
all results proven in this paper can be shown true no matter what characterization is
considered, and we will not specify a subscript \(i\in\{1,2,3\}\).
For references cited, while the distinctions can generally be hand-waved away,
the reader should be aware that most authors consider at most one of these
three characterizations.

\section{Background and Preliminaries} \label{section:Background}

We will use the standard definition of \(\omega\) where \(n \in \omega\)
is \(\{ m \in \omega : m \in n \}\).
Hence, given \(A \subseteq \omega\) and \(n \in \omega\),
we may write \(A \subseteq n\).
We let \([X]^{<\omega}\) denote the set of all finite subsets of a set \(X\).
We will use \(\pi_X : X \times Y \to X\) and \(\pi_Y : X \times Y \to Y\)
to denote the usual coordinate projection mappings.

We will use \(\mathcal O_X\) to denote the collection of all open covers of \(X\),
viewing \(\mathcal O\) as a \emph{topological operator}.
A topological operator is a class function defined on the class of all topological spaces.
Another topological operator that will appear here is \(\mathscr T\), the topological operator that produces
all non-empty open subsets of a space \(X\).

When a topological space \(X\) is given and \(A \subseteq X\), we denote the neighborhood system
\(\{ U \in \mathscr T_X : A \subseteq U \}\) about \(A\) with \(\mathcal N_A\).
For \(x \in X\), we use the simplified notation \(\mathcal N_x\) instead of \(\mathcal N_{\{x\}}\),
if there is no risk of confusion.
Also, for a collection \(\mathcal A\) of subsets of \(X\),
\(\mathcal N(\mathcal A) = \{ \mathcal N_A : A \in \mathcal A \}\).

We will consider two other kinds of open covers.
\begin{definition}
    For a space \(X\), an open cover \(\mathscr U\) of \(X\) is said to be
    \begin{itemize}
        \item
        an \emph{\(\omega\)-cover} of \(X\) if every finite subset of \(X\) is contained in a member of \(\mathscr U\).
        \item 
        a \emph{\(k\)-cover} of \(X\) if every compact subset of \(X\) is contained in a member of \(\mathscr U\).
    \end{itemize}
    We will let \(\Omega\) (resp. \(\mathcal K\)) be the topological operator which produces \(\Omega_X\) (resp. \(\mathcal K_X\)), the set of all
    \(\omega\)-covers (resp. \(k\)-covers) of a space \(X\).
\end{definition}

The notion of \(\omega\)-covers is commonly attributed to \cite{GerlitsNagy}, but they were already in use in \cite{McCoyOmegaCovers}
where they are refereed to as \emph{open covers for finite sets}.
The notion of \(k\)-covers appears as early as \cite{McCoyKcovers} in which they are referred to as \emph{open covers for compact subsets.}

We remind the reader of a generalization of the above-mentioned cover types.
\begin{definition}
    Let \(\mathcal A\) be a collection of subsets of a space \(X\).
    Then we define the \(\mathcal A\)-covers, denoted by \(\mathcal O_X(\mathcal A)\),
    to be the collection of all open covers \(\mathscr U\) of \(X\)
    such that, for each \(A \in \mathcal A\), there is some \(U \in \mathscr U\) such that \(A \subseteq U\).
\end{definition}

We recall the usual selection principles.
For more details on selection principles and relevant references, see \cite{ScheepersI,KocinacSelectedResults,ScheepersSelectionPrinciples,ScheepersNoteMat}.
\begin{definition}
    Let \(\mathcal A\) and \(\mathcal B\) be sets.
    Then the single- and finite-selection principles are defined, respectively, to be the properties
    \[\mathsf S_1(\mathcal A, \mathcal B) \equiv 
    \left(\forall A \in \mathcal A^\omega\right)\left(\exists B \in \prod_{n \in \omega} A_n\right)\ \{B_n : n \in \omega\} \in \mathcal B\]
    and
    \[\mathsf S_{\mathrm{fin}}(\mathcal A, \mathcal B) \equiv 
    \left(\forall A \in \mathcal A^\omega\right)\left(\exists B \in \prod_{n \in \omega} [A_n]^{<\omega}\right)\ \bigcup\{B_n : n \in \omega\} \in \mathcal B.\]
    Following \cite{ScheepersNoteMat}, for a space \(X\) and topological operators \(\mathcal A\) and \(\mathcal B\),
    we write \(X \models \mathsf S_\square(\mathcal A, \mathcal B)\), where \(\square \in \{ 1 , \mathrm{fin} \}\),
    to mean that \(X\) satisfies the selection principle \(\mathsf S_\square(\mathcal A_X, \mathcal B_X)\).
\end{definition}
Using this notation, recall that a space \(X\) is \emph{Menger} (resp. \emph{Rothberger}) if \(X \models \mathsf S_{\mathrm{fin}}(\mathcal O, \mathcal O)\)
(resp. \(X \models \mathsf S_1(\mathcal O, \mathcal O)\)).

Selection principles have naturally corresponding selection games, which include types of topological games.
Topological games have a long history, much of which can be gathered from Telg{\'a}rsky's survey
\cite{TelgarskySurvey}.
In this paper, we consider the traditional selection games for two players, P1 and P2, of countably infinite length.
\begin{definition}
    Given sets \(\mathcal A\) and \(\mathcal B\), we define the \emph{finite-selection game}
    \(\mathsf{G}_{\mathrm{fin}}(\mathcal A, \mathcal B)\) for \(\mathcal A\) and \(\mathcal B\) as follows.
    In round \(n \in \omega\), P1 plays \(A_n \in \mathcal A\) and P2 responds with \(\mathscr F_n \in [A_n]^{<\omega}\).
    We declare P2 the winner if \(\bigcup\{ \mathscr F_n : n \in \omega \} \in \mathcal B\).
    Otherwise, P1 wins.
\end{definition}
\begin{definition}
    Given sets \(\mathcal A\) and \(\mathcal B\), we analogously define the \emph{single-selection game}
    \(\mathsf{G}_{1}(\mathcal A, \mathcal B)\) for \(\mathcal A\) and \(\mathcal B\) as follows.
    In round \(n \in \omega\), P1 plays \(A_n \in \mathcal A\) and P2 responds with \(x_n \in A_n\).
    We declare P2 the winner if \(\{x_n : n \in \omega \} \in \mathcal B\).
    Otherwise, P1 wins.
\end{definition}
\begin{definition}
    By \emph{selection games}, we mean the class consisting of \(\mathsf G_\square(\mathcal A, \mathcal B)\)
    where \(\square \in \{1, \mathrm{fin} \}\), and \(\mathcal A\) and \(\mathcal B\) are sets.
    So, when we say \(\mathcal G\) is a selection game, we mean that there exist \(\square \in \{1 , \mathrm{fin} \}\)
    and sets \(\mathcal A, \mathcal B\) so that \(\mathcal G = \mathsf G_\square(\mathcal A, \mathcal B)\).
\end{definition}
The study of games naturally inspires questions about the existence of various kinds of strategies.
Infinite games and corresponding full-information strategies were both introduced in \cite{GaleStewart}.
Some forms of limited-information strategies came shortly after, like positional (also known as stationary) strategies \cite{DavisGames,SchmidtGames}.
For more on stationary and Markov strategies, see \cite{GalvinTelgarsky}.
\begin{definition}
    We define strategies of various strengths below.
    \begin{itemize}
    \item
    A \emph{strategy for P1} in \(\mathsf{G}_1(\mathcal A, \mathcal B)\) is a function
    \(\sigma:(\bigcup \mathcal A)^{<\omega} \to \mathcal A\).
    A strategy \(\sigma\) for P1 is called \emph{winning} if whenever \(x_n \in \sigma\langle x_k : k < n \rangle\)
    for all \(n \in \omega\), \(\{x_n: n\in\omega\} \not\in \mathcal B\).
    If P1 has a winning strategy, we write \(\mathrm{I} \uparrow \mathsf{G}_1(\mathcal A, \mathcal B)\).
    \item
    A \emph{strategy for P2} in \(\mathsf{G}_1(\mathcal A, \mathcal B)\) is a function
    \(\sigma:\mathcal A^{<\omega} \to \bigcup \mathcal A\).
    A strategy \(\sigma\) for P2 is \emph{winning} if whenever \(A_n \in \mathcal A\) for all \(n \in \omega\),
    \(\{\sigma(A_0,\ldots,A_n) : n \in \omega\} \in \mathcal B\).
    If P2 has a winning strategy, we write \(\mathrm{II} \uparrow \mathsf{G}_1(\mathcal A, \mathcal B)\).
    \item
    A \emph{predetermined strategy} for P1 is a strategy which only considers the current turn number.
    Formally it is a function \(\sigma: \omega \to \mathcal A\).
    If P1 has a winning predetermined strategy, we write \(\mathrm{I} \underset{\mathrm{pre}}{\uparrow} \mathsf{G}_1(\mathcal A, \mathcal B)\).
    \item
    A \emph{Markov strategy} for P2 is a strategy which only considers the most recent move of P1 and the current turn number.
    Formally it is a function \(\sigma :\mathcal A \times \omega \to \bigcup \mathcal A\).
    If P2 has a winning Markov strategy, we write
    \(\mathrm{II} \underset{\mathrm{mark}}{\uparrow} \mathsf{G}_1(\mathcal A, \mathcal B)\).
    \item
    If there is a single element \(A_0 \in \mathcal A\) so that the constant function with value \(A_0\) is a winning strategy for P1, we say that P1 has a \emph{constant winning strategy}, denoted by \(\mathrm{I} \underset{\mathrm{cnst}}{\uparrow} \mathsf{G}_1(\mathcal A, \mathcal B)\).
    \end{itemize}
    These definitions can be extended to \(\mathsf{G}_{\mathrm{fin}}(\mathcal A, \mathcal B)\) in the obvious way.
\end{definition}
Note that, for any selection game \(\mathcal G\),
\[
    \mathrm{II} \underset{\mathrm{mark}}{\uparrow} \mathcal G
    \implies \mathrm{II} \uparrow \mathcal G
    \implies \mathrm{I} \not\uparrow \mathcal G
    \implies \mathrm{I} \underset{\mathrm{pre}}{\not\uparrow} \mathcal G
    \implies \mathrm{I} \underset{\mathrm{cnst}}{\not\uparrow} \mathcal G.
\]
\begin{definition}
    For two selection games \(\mathcal G\) and \(\mathcal H\), we write \(\mathcal G \leq_{\mathrm{II}} \mathcal H\)
    if each of the following hold:
    \begin{itemize}
        \item 
        \(\mathrm{II} \underset{\mathrm{mark}}{\uparrow} \mathcal G \implies \mathrm{II} \underset{\mathrm{mark}}{\uparrow} \mathcal H\)
        \item 
        \(\mathrm{II} \uparrow \mathcal G \implies \mathrm{II} \uparrow \mathcal H\)
        \item 
        \(\mathrm{I} \not\uparrow \mathcal G \implies \mathrm{I} \not\uparrow \mathcal H\)
        \item 
        \(\mathrm{I} \underset{\mathrm{pre}}{\not\uparrow} \mathcal G \implies \mathrm{II} \underset{\mathrm{pre}}{\not\uparrow} \mathcal H\)
    \end{itemize}
    If, in addition,
    \begin{itemize}
        \item 
        \(\mathrm{I} \underset{\mathrm{cnst}}{\not\uparrow} \mathcal G \implies \mathrm{II} \underset{\mathrm{cnst}}{\not\uparrow} \mathcal H,\)
    \end{itemize}
    we write that \(\mathcal G \leq^+_{\mathrm{II}} \mathcal H\).
\end{definition}
Note that, for any sets \(\mathcal A\) and \(\mathcal B\),
\[\mathsf G_1(\mathcal A, \mathcal B) \leq^+_{\mathrm{II}} \mathsf G_{\mathrm{fin}}(\mathcal A, \mathcal B).\]

We use the notation \(\leq_{\mathrm{II}}\) to emphasize the fact that this partial order transfers winning
plays for P2.
As an example, note that \(\mathrm{I} \not\uparrow \mathcal G\) means that, for any strategy that P1 employs,
there exists a play by P2 that wins against that strategy.
Then the implication in the definition of \(\leq_{\mathrm{II}}\) would indicate that P2 can accordingly
win against any strategy employed by P1 in \(\mathcal H\).

Below, we will write statements such as \(\mathsf G_\square(\mathcal A, \mathcal A) \leq_{\mathrm{II}} \mathsf G_\square(\mathcal B, \mathcal B)\)
where \(\mathcal A\) and \(\mathcal B\) are topological operators to mean that, for every topological space \(X\),
\(\mathsf G_\square(\mathcal A_X, \mathcal A_X) \leq_{\mathrm{II}} \mathsf G_\square(\mathcal B_X, \mathcal B_X)\).

\begin{remark} \label{remark:LindelofAndSelection}
    The following are mentioned in \cite[Prop. 15]{ClontzDualSelection} and \cite[Lem. 2.12]{CHVietoris}
    for \(\square \in \{1,\mathrm{fin}\}\).
    \begin{itemize}
        \item 
        \(\mathrm{I} \underset{\mathrm{pre}}{\not\uparrow} \mathsf{G}_\square(\mathcal{A},\mathcal{B})\)
        is equivalent to \(\mathsf{S}_\square(\mathcal A, \mathcal B)\).
        \item
        \(\mathrm{I} \underset{\mathrm{cnst}}{\not\uparrow} \mathsf{G}_\square(\mathcal{A},\mathcal{B})\)
        is equivalent to the property that, for every \(A \in \mathcal A\), there is \(B \in [A]^{\leq \omega}\)
        so that \(B \in \mathcal B\).
    \end{itemize}
    Note that the property \(\mathrm{I} \underset{\mathrm{cnst}}{\not\uparrow} \mathsf{G}_\square(\mathcal{A},\mathcal{B})\)
    is a Lindel{\"{o}}f-like principle and falls in the category of what Scheepers \cite{ScheepersNoteMat}
    refers to as \emph{Bar-Ilan selection principles}.
\end{remark}
In particular, note that, if \(\mathsf G_\square(\mathcal A, \mathcal B) \leq_{\mathrm{II}} \mathsf G_\square(\mathcal C, \mathcal D)\)
where \(\square \in \{1, \mathrm{fin}\}\), then, for any space \(X\),
\(X \models \mathsf S_\square(\mathcal A, \mathcal B) \implies X \models \mathsf S_\square(\mathcal C, \mathcal D)\).

Following \cite{KocinacSelectedResults}, we will employ the following terminology.
A space \(X\) is
\begin{itemize}
    \item 
    \emph{\(\omega\)-Lindel\"{o}f} (referred to as \(\epsilon\)-spaces in \cite{GerlitsNagy}) if every \(\omega\)-cover
    has a countable subset which is an \(\omega\)-cover; equivalently, if \(\mathrm{I} \underset{\mathrm{cnst}}{\not\uparrow}
    \mathsf G_{\mathrm{fin}}(\Omega_X,\Omega_X)\).
    \item 
    \emph{\(k\)-Lindel\"{o}f} if every \(k\)-cover has a countable subset which is a \(k\)-cover;
    equivalently, if \(\mathrm{I} \underset{\mathrm{cnst}}{\not\uparrow}
    \mathsf G_{\mathrm{fin}}(\mathcal K_X,\mathcal K_X)\).
\end{itemize}
\begin{proposition} \label{prop:SecondCountable}
    Every second-countable space is \(k\)-Lindel\"{o}f and \(\omega\)-Lindel\"{o}f.
\end{proposition}
\begin{proof}
    Let \(\mathscr B_0\) be a countable basis for a space \(X\) and let
    \[\mathscr B = \left\{ \bigcup \mathscr F : \mathscr F \in \left[\mathscr B_0\right]^{<\omega} \right\}.\]
    Notice that \(\mathscr B\) is also countable.
    Since we will establish with Theorem \ref{thm:KImpliesOmegaImpliesOpen} that every \(k\)-Lindel\"{o}f
    space is \(\omega\)-Lindel\"{o}f, we prove here that \(X\) is \(k\)-Lindel\"{o}f,
    though the direct proof of \(\omega\)-Lindel\"{o}fness in this case is identical in form to what follows.
    
    So let \(\mathscr U\) be a \(k\)-cover of \(X\), and, for each compact \(K \subseteq X\), let \(U \in \mathscr U\)
    be such that \(K \subseteq U\).
    We can then find \(\mathscr F_K \in \left[ \mathscr B_0 \right]^{<\omega}\) such that \(K \subseteq \bigcup \mathscr F_K \subseteq U\).
    Note then that \(\bigcup \mathscr F_K \in \mathscr B\).
    Now, \[\mathscr V := \left\{ \bigcup \mathscr F_K : K \subseteq X \text{ is compact} \right\} \subseteq \mathscr B,\]
    and is thus countable.
    Make a choice \(U_V \in \mathscr U\) for each \(V \in \mathscr V\) with \(V \subseteq U_V\)
    and note that \(\{ U_V : V \in \mathscr V \}\) is the desired countable subset of \(\mathscr U\).
\end{proof}

We recall some P1 strategy reduction theorems,
the first of which are the celebrated theorems of Hurewicz \cite{HurewiczOriginal}
and Pawlikowski \cite{Pawlikowski}.
For more on the proofs of the theorems of Hurewicz and Pawlikowski, we refer the reader to \cite{ConceptualProofs};
for a pointless (that is, lattice-theoretic) approach, see \cite{Mezabarba}.

\begin{theorem}[{Hurewicz \cite{HurewiczOriginal}/Pawlikowski \cite{Pawlikowski}}]
    For \(\square \in \{1 , \mathrm{fin} \}\),
    \[\mathrm{I} \uparrow \mathsf G_\square(\mathcal O, \mathcal O)
    \iff \mathrm{I} \underset{\mathrm{pre}}{\uparrow} \mathsf G_\square(\mathcal O, \mathcal O).\]
\end{theorem}
\begin{theorem}[{Scheepers \cite{ScheepersIII}}]
    For \(\square \in \{1 , \mathrm{fin} \}\),
    \[\mathrm{I} \uparrow \mathsf G_\square(\Omega, \Omega)
    \iff \mathrm{I} \underset{\mathrm{pre}}{\uparrow} \mathsf G_\square(\Omega, \Omega).\]
\end{theorem}
\begin{theorem}[{Caruvana {\&} Holshouser \cite{CHVietoris}}]
    For \(\square \in \{1 , \mathrm{fin} \}\),
    \[\mathrm{I} \uparrow \mathsf G_\square(\mathcal K, \mathcal K)
    \iff \mathrm{I} \underset{\mathrm{pre}}{\uparrow} \mathsf G_\square(\mathcal K, \mathcal K).\]
\end{theorem}
With these results, the properties considered in this paper ``collapse'' for P1 and we need only have terminology
for the corresponding selection principles.
So, following the notation above, we will say a space \(X\) is
\begin{itemize}
    \item 
    \emph{\(\omega\)-Menger} (resp. \emph{\(\omega\)-Rothberger}) if \(X \models \mathsf S_{\mathrm{fin}}(\Omega, \Omega)\)
    (resp. \(X \models \mathsf S_1(\Omega,\Omega)\)).
    \item 
    \emph{\(k\)-Menger} (resp. \emph{\(k\)-Rothberger}) if \(X \models \mathsf S_{\mathrm{fin}}(\mathcal K, \mathcal K)\)
    (resp. \(X \models \mathsf S_1(\mathcal K,\mathcal K)\)).
\end{itemize}

We will, however, have distinguishing terminology for the situation on P2's side of things.
\begin{definition}
    For the properties \(\mathcal P\) discussed above, we will say that \(X\) is \emph{strategically \(\mathcal P\)}
    if P2 has a winning strategy in the game corresponding to \(\mathcal P\).
    Analogously, we will say a space is \emph{Markov} \(\mathcal P\) if P2 has a winning Markov strategy
    in the corresponding game to \(\mathcal P\).
\end{definition}

\section{General Results} \label{sec:GeneralResults}

\subsection{Characterizations using finite powers} \label{subsection:FinitePowers}

An important characterization of the \(\omega\)-variants is found in finite powers.
We will collect these characterizations in each of their strategic levels.
\begin{theorem}[{\cite[p. 156]{GerlitsNagy}}] \label{thm:GerlitsNagy}
    A space \(X\) is \(\omega\)-Lindel\"{o}f if and only if each of its finite powers is Lindel\"{o}f.
\end{theorem}
\begin{theorem} \label{thm:OmegaSelectionChar}
    For any space \(X\),
    \begin{enumerate}[label=(\arabic*),ref=(\arabic*)]
        \item \label{OmegaSelectionMengerChar}
        \(X\) is \(\omega\)-Menger if and only if each of its finite powers is Menger (\cite[Thm. 3.9]{CombOpenCovers2}).
        \item \label{OmegaSelectionRothbergerChar}
        \(X\) is \(\omega\)-Rothberger if and only if each of its finite powers is Rothberger (\cite[p. 918]{Sakai1988}).
    \end{enumerate}
\end{theorem}

To characterize the ``strategically \(\mathcal P\)'' level without any separation axiom assumptions,
we recall that being strategically Rothberger (resp. Menger) is finitely productive.
The single-selection version of the equivalence referred to here will appear here as Theorem \ref{thm:StrategicOmega}\ref{StrategicRothChar}
and was originally proved in \cite{ClontzHolshouser} by passing through a related game on \(C_p(X)\)
under the assumption that \(X\) is Tychonoff.
We will see that one direction of the general equivalence will be the result of Theorem \ref{thm:DiasScheepers} and
a bijection \(\omega^2 \to \omega\).

\begin{theorem} \label{thm:DiasScheepers}
    Let \(X\) and \(Y\) be spaces.
    \begin{enumerate}[label=(\arabic*),ref=(\arabic*)]
        \item \label{DiasScheepersMenger}
        If \(X\) and \(Y\) are strategically Menger, then \(X \times Y\) is strategically Menger (\cite[Prop. 5.1]{DiasScheepers}).
        \item \label{DiasScheepersRothberger}
        If \(X\) and \(Y\) are strategically Rothberger, then \(X \times Y\) is strategically Rothberger (\cite[Cor. 4.9]{DiasScheepers}).
    \end{enumerate}
\end{theorem}
As an immediate consequence, if a space \(X\) is strategically Menger or strategically Rothberger, then each of its finite powers
is, too.

To obtain the reverse direction of Theorem \ref{thm:StrategicOmega}\ref{StrategicRothChar},
we use a generalized idea of Galvin \cite{Galvin1978} below, which is proved by a straightforward diagonalization,
and a basic unfolding argument.

\begin{lemma} \label{lem:OpenCoverChoices}
    Let \(X\) be a space, \(\mathcal A\) be a collection of subsets of \(X\),
    and suppose \(\varphi : \mathcal O_X(\mathcal A) \to \mathscr T_X\).
    If \(\varphi(\mathscr U) \in \mathscr U\) for every \(\mathscr U \in \mathcal O_X(\mathcal A)\),
    then there exists \(A \in \mathcal A\) such that \(\mathcal N_A \subseteq \varphi[\mathcal O_X(\mathcal A)]\).
\end{lemma}
\begin{proof}
    By way of contrapositive, suppose that, for every \(A \in \mathcal A\), there exists \(U \in \mathcal N_A\)
    such that \(U \not\in \varphi[\mathcal O_X(\mathcal A)]\).
    So let \(U_A \in \mathcal N_A\) witness this property for each \(A \in \mathcal A\).
    Note that \(\mathscr U := \{ U_A : A \in \mathcal A \} \in \mathcal O_X(\mathcal A)\).
    It follows, by construction, that \(\varphi(\mathscr U) \not\in \mathscr U\).
\end{proof}
\begin{theorem} \label{thm:StrategicOmega}
    Let \(X\) be a space.
    \begin{enumerate}[label=(\arabic*),ref=(\arabic*)]
        \item \label{StrategicMengerChar}
        \(X\) is strategically \(\omega\)-Menger if and only if it is strategically Menger.
        \item \label{StrategicRothChar}
        \(X\) is strategically \(\omega\)-Rothberger if and only it is strategically Rothberger.
    \end{enumerate}
\end{theorem}
\begin{proof}
    The content of \ref{StrategicMengerChar} is \cite[Thm. 35]{ClontzRelatingGamesOfMenger}.

    Now we address \ref{StrategicRothChar}.
    Suppose \(X\) is strategically Rothberger.
    By Theorem \ref{thm:DiasScheepers}\ref{DiasScheepersRothberger},
    \(X^{n+1}\) is strategically Rothberger for every \(n \in \omega\).
    So let \(\sigma_n\) be a winning strategy for P2 in the Rothberger game on \(X^{n+1}\)
    for each \(n \in \omega\), and let \(\beta : \omega^2 \to \omega\) be a bijection
    with the property that \(\langle \beta(n,k) : k \in \omega\rangle\) is strictly increasing
    for \(n\in\omega\).
    Now define the strategy \(\sigma\) in the \(\omega\)-Rothberger game on \(X\) in the following way.
    Given \(n \in \omega\) and a sequence \(\langle \mathscr U_j : j < n \rangle\) of \(\omega\)-covers of \(X\),
    let \((m,k) \in \omega^2\) be such that \(\beta(m,k) = n\).
    Note that, for each \(\ell \leq k\),
    \[\mathscr U_{\beta(m,\ell)}^{(m+1)} := \left\{ U^{m+1} : U \in \mathscr U_{\beta(m,\ell)} \right\} \in \mathcal O_{X^{m+1}}.\]
    So we can define
    \[\sigma(\langle \mathscr U_j : j \leq n \rangle) =
    \pi_{m+1} \left[\sigma_m\left(\left\langle \mathscr U_{\beta(m,\ell)}^{(m+1)} : \ell \leq k \right\rangle \right)\right],\]
    where \(\pi_{m+1}\) is the projection mapping.
    From here, it's straightforward to check that \(\sigma\) is winning.

    Now we assume that \(X\) is strategically \(\omega\)-Rothberger.
    So let \(\sigma_0\) be a winning strategy for P2 in the \(\omega\)-Rothberger
    game.
    We define a strategy \(\sigma\) for P2 in the Rothberger game as follows.
    
    By Lemma \ref{lem:OpenCoverChoices}, we let
    \[F_0 \in [X]^{<\omega}\]
    be such that
    \[\mathcal N_{F_0} \subseteq \left\{ \sigma_0(\mathscr W) : \mathscr W \in \Omega_X \right\}.\]
    Let \(M_0 = -1 + \# F_0\) and enumerate it as
    \(F_0 = \{x_0, \ldots, x_{M_0}\}\).
    Now suppose \(\langle \mathscr U_j : j \leq M_0 \rangle\) is a given
    sequence of open covers of \(X\).
    For each \(k \leq M_0\), choose \(\sigma(\langle \mathscr U_j : j \leq k \rangle) \in \mathscr U_k\)
    to be such that
    \[x_k \in \sigma(\langle \mathscr U_j : j \leq k \rangle).\]
    Then, we can choose
    \(\mathscr W_0 \in \Omega_X\)
    to be such that
    \[\sigma_0\left( \mathscr W_0 \right)
    = \bigcup_{k=0}^{M_0} \sigma(\langle \mathscr U_j : j \leq k \rangle).\]
    
    Now let \(n \in \omega\) be given and suppose we have
    \(\langle F_j : j \leq n \rangle\), \(\langle M_j : j \leq n \rangle\),
    \(\langle \mathscr W_j : j \leq n \rangle\), and \(\langle \mathscr U_j : j \leq M_n \rangle\)
    defined.
    As above, we can let \(F_{n+1} \in [X]^{<\omega}\) be such that
    \[\mathcal N_{F_{n+1}}
    \subseteq \left\{\sigma_0\left( \mathscr W_0, \ldots, \mathscr W_n , \mathscr W \right)
    : \mathscr W \in \Omega_X \right\}.\]
    Let \(M_{n+1} = M_n + \# F_{n+1}\) and enumerate \(F_{n+1}\)
    as \(\{ x_{M_n+1} , \ldots , x_{M_{n+1}} \}\).
    Then, given a sequence \[\langle \mathscr U_j : M_n < j \leq M_{n+1} \rangle\]
    of open covers of \(X\), we define, for \(M_n < k \leq M_{n+1}\),
    \[\sigma(\langle \mathscr U_j : j \leq k \rangle) \in \mathscr U_k\]
    to be such that
    \[x_k \in \sigma(\langle \mathscr U_j : j \leq k \rangle).\]
    Then, we set \(\mathscr W_{n+1} \in \Omega_X\) to be such that
    \[\sigma_0(\mathscr W_0, \ldots , \mathscr W_{n+1})
    = \bigcup_{k=M_n+1}^{M_{n+1}} \sigma(\langle \mathscr U_j : j \leq k \rangle).\]
    This defines \(\sigma\).
    
    To see that \(\sigma\) is winning, consider a play \(\langle \mathscr U_n : n \in \omega\rangle\)
    of the game according to \(\sigma\) and let \(x \in X\) be arbitrary.
    Since \(\sigma_0\) is winning in the \(\omega\)-Rothberger game, there is some
    \(n \in \omega\) such that
    \[\{x\} \subseteq \sigma_0(\mathscr W_0, \ldots, \mathscr W_n)
    = \bigcup_{k=M_n+1}^{M_{n+1}} \sigma(\langle \mathscr U_j : j \leq k \rangle).\]
    Hence, there is some \(M_n < k \leq M_{n+1}\) such that
    \[x \in \sigma(\mathscr U_0, \ldots, \mathscr U_k).\]
    That is, \(\sigma\) is winning.
\end{proof}

To characterize the Markov properties of interest in this paper, we will need some other notions.
\begin{definition}
    A space \(X\) is said to be \emph{topologically countable} if there exists
    \(\{ x_n : n \in \omega \} \subseteq X\) such that
    \(X = \bigcup_{n\in\omega} \bigcap \mathcal N_{x_n}\).
\end{definition}
\begin{lemma} \label{lem:BasicT1}
    If \(X\) is a \(T_1\) space, then, for \(A \subseteq X\), \(A = \bigcap \mathcal N_A\).
    Consequently, any \(T_1\) space that is topologically countable is countable.
\end{lemma}
\begin{proof}
    Clearly \(A \subseteq \bigcap \mathcal N_A\).
    On the other hand, for \(x \not\in A\), \(X \setminus \{x\} \in \mathcal N_A\)
    since \(X\) is \(T_1\).
    Hence, \(x \not\in \bigcap \mathcal N_A\).
\end{proof}
\begin{example}[\href{https://topology.pi-base.org/spaces/S000042}{S42}\footnote{Throughout, we will refer to spaces by their ID in the \(\pi\)-base \cite{PiBase}.}]
    The right-ordered reals \(X\) is an example of an uncountable \(T_0\) space which is topologically countable.
    Indeed, since the basis for \(X\) consists of intervals \((a,\infty)\) for \(a \in \mathbb R\), the set of integers witnesses
    topological countability.
\end{example}
\begin{definition}
    For a subset \(A\) of a space \(X\), \(A\) is said to be \emph{relatively compact} (in \(X\)) if every open cover of \(X\)
    admits a finite subset which covers \(A\).
    A space \(X\) is said to be \emph{\(\sigma\)-relatively compact} if there exists a countable collection
    \(\{ A_n : n \in \omega \}\) of relatively compact subsets of \(X\) such that \(X = \bigcup_{n\in\omega} A_n\).
\end{definition}
The usual Tube Lemma idea applies to show that the property of relative compactness is also productive.
\begin{lemma} \label{lemma:ProductOfRelativelyCompact}
    If \(A\) and \(B\) are relatively compact subsets of \(X\) and \(Y\), respectively, then \(A \times B\) is relatively
    compact in \(X \times Y\).
\end{lemma}
\begin{proof}
    Without loss of generality, we consider only basic covers.
    So let \(\mathscr W\) be an open cover of \(X\times Y\) consisting of rectangles.
    Note that, for any \(x\in X\), \(\{x\}\times B\) is relatively compact in \(\{x\} \times Y\)
    viewed as a subspace of \(X \times Y\).
    So, for each \(x\in X\), we can let \(\mathscr F_x \in [\mathscr W]^{<\omega}\) be such that
    \(\{ x \} \times B \subseteq \bigcup \mathscr F_x\).
    Then let \(U_x = \bigcap \{ \pi_X[W] : W \in \mathscr F_x \}\).
    Now, \(\{ U_x : x \in X \}\) is an open cover of \(X\) so we can find \(F \in [X]^{<\omega}\)
    such that \(A \subseteq \bigcup \{ U_x : x \in F \}\).
    Define \(\mathscr W_0 = \bigcup_{x \in F} \mathscr F_x\).

    To finish the proof, we need only show that \(A \times B \subseteq \bigcup \mathscr W_0\).
    So let \(\langle x, y \rangle \in A \times B\) be arbitrary.
    There is some \(a \in F\) so that \(x \in U_a\).
    Since \(y \in B\), there is some \(W \in \mathscr F_a\) so that \(\langle a, y \rangle \in W\).
    Note then that \(x \in U_a \subseteq \pi_X[W]\) which establishes that \(\langle x , y \rangle \in W\).
\end{proof}
\begin{lemma} \label{lem:RelCompactRegular}
    The closure of any relatively compact subset in a regular space is compact.
    (See \cite{ArhangelskiiClassic} and also \cite[Prop. 4.4]{ClontzApplicationsOfLimitedInformation}.)
\end{lemma}
An immediate consequence is
\begin{corollary} \label{cor:RegularSigmaRelCompac}
    Every regular \(\sigma\)-relatively compact space is \(\sigma\)-compact.
\end{corollary}
\begin{example}[\href{https://topology.pi-base.org/spaces/S000059}{S59}]
    The indiscrete irrational extension of the reals is \(\sigma\)-relatively compact but not \(\sigma\)-compact (see
    \cite[Ex. 5.8]{ClontzApplicationsOfLimitedInformation}).
\end{example}

With eyes toward generality, we define the following types of cofinality.
\begin{definition}
    Let \(X\) be a set and suppose \(\mathcal A\) and \(\mathcal B\) are collections of subsets of \(X\).
    We say that \(\mathrm{cof}_X(\mathcal A, \mathcal B) \leq \omega\) if there exists
    \(\{ A_n : n \in \omega\} \subseteq \mathcal A\) such that, for every \(B \in \mathcal B\), there exists \(n \in \omega\)
    such that \(B \subseteq A_n\).
\end{definition}
In words, the condition \(\mathrm{cof}_X(\mathcal A, \mathcal B) \leq \omega\) is asserting that
\(\mathcal A\) is of cofinality type \(\omega\) relative to \(\mathcal B\).

We also use a slightly weaker version, inspired by the notion of being topologically countable.
\begin{definition}
    Let \(X\) be a space and suppose \(\mathcal A\) and \(\mathcal B\) are collections of subsets of \(X\).
    We say that \(\widehat{\mathrm{cof}}_X(\mathcal A, \mathcal B) \leq \omega\) if there exists
    \(\{ A_n : n \in \omega\} \subseteq \mathcal A\) such that, for every \(B \in \mathcal B\), there exists \(n \in \omega\)
    such that \(B \subseteq \bigcap \mathcal N_{A_n}\).
\end{definition}

Note that, if we identify \(X\) with the set of singletons and let \(\mathcal A\) be the collection
of relatively compact subsets of \(X\), then \(X\) is
\begin{itemize}
    \item 
    \(\sigma\)-relatively compact if and only if \(\mathrm{cof}_X(\mathcal A , X) \leq \omega\).
    \item 
    topologically countable if and only if \(\widehat{\mathrm{cof}}_X(X,X) \leq \omega\).
\end{itemize}
\begin{lemma} \label{lem:NearlyCofinal}
    For any space \(X\) and a collection \(\mathcal A\) of subsets of \(X\),
    \(\widehat{\mathrm{cof}}_X(\mathcal A, \mathcal A) \leq \omega\)
    if and only if \(\mathrm{II} \underset{\mathrm{mark}}{\uparrow} \mathsf G_1(\mathcal O_X(\mathcal A), \mathcal O_X(\mathcal A))\).
\end{lemma}
\begin{proof}
    Let \(\{ A_n : n \in \omega \} \subseteq \mathcal A\) witness that
    \(\widehat{\mathrm{cof}}_X(\mathcal A, \mathcal A) \leq \omega\).
    For each \(n \in \omega\) and \(\mathscr U \in \mathcal O_X(\mathcal A)\),
    choose \(\sigma(\mathscr U,n) \in \mathscr U\)
    such that \(A_n \subseteq \sigma(\mathscr U, n)\).
    To see that \(\sigma\) is winning, let \(\{ \mathscr U_n : n \in \omega \} \subseteq \mathcal O_X(\mathcal A)\)
    and \(A \in \mathcal A\).
    Choose \(n \in \omega\) to be such that
    \(A \subseteq \bigcap \mathcal N_{A_n}\).
    Then \[A \subseteq \bigcap \mathcal N_{A_n} \subseteq \sigma(\mathscr U_n, n).\]
    Hence, \(\{ \sigma(\mathscr U_n,n) : n \in \omega \} \in \mathcal O_X(\mathcal A)\).

    Now suppose that \(\widehat{\mathrm{cof}}_X(\mathcal A, \mathcal A) \not\leq \omega\)
    and let \(\sigma\) be a Markov strategy for P2 in \(\mathsf G_1(\mathcal O_X(\mathcal A), \mathcal O_X(\mathcal A))\).
    By Lemma \ref{lem:OpenCoverChoices}, we can choose, for each \(n \in \omega\), \(A_n \in \mathcal A\)
    such that \[\mathcal N_{A_n} \subseteq \{\sigma(\mathscr U, n) : \mathscr U \in \mathcal O_X(\mathcal A) \}.\]
    Now, by the assumption, there is some \(A \in \mathcal A\) such that
    \(A \setminus \bigcap \mathcal N_{A_n} \neq \emptyset\) for each \(n \in \omega\).
    Hence, for each \(n \in \omega\), there is some \(\mathscr U_n \in \mathcal O_X(\mathcal A)\)
    such that \(A \setminus \sigma(\mathscr U_n,n) \neq \emptyset\).
    Note then that \[\{ \sigma(\mathscr U_n,n) : n \in \omega \} \not\in \mathcal O_X(\mathcal A).\]
    That is, \(\sigma\) is not winning.
\end{proof}

\begin{theorem} \label{thm:MarkovOmega}
    Let \(X\) be a space.
    \begin{enumerate}[label=(\arabic*),ref=(\arabic*)]
        \item  \label{thm:MarkovOmegaMenger}
        The following are equivalent.
        \begin{enumerate}
            \item 
            \(X\) is Markov \(\omega\)-Menger.
            \item
            \(X\) is Markov Menger.
            \item
            \(X\) is \(\sigma\)-relatively compact.
        \end{enumerate}
        \item \label{thm:MarkovOmegaRothberger}
        The following are equivalent.
        \begin{enumerate}
            \item 
            \(X\) is Markov \(\omega\)-Rothberger.
            \item
            \(X\) is Markov Rothberger.
            \item
            \(X\) is topologically countable.
        \end{enumerate}
    \end{enumerate}
\end{theorem}
\begin{proof}
    We start by addressing \ref{thm:MarkovOmegaMenger}.
    The fact that \(X\) is \(\sigma\)-relatively compact if and only if \(X\) is Markov Menger
    is \cite[Cor. 4.7]{ClontzApplicationsOfLimitedInformation}.
    Also note that the \(X\) being Markov \(\omega\)-Menger immediately implies that \(X\)
    is Markov Menger by considering the closure under finite unions of open covers of \(X\).
    So, to finish this portion of the proof, we show that being \(\sigma\)-relatively compact
    guarantees that the space is Markov \(\omega\)-Menger.
    So suppose \(X\) is \(\sigma\)-relatively compact.
    By Lemma \ref{lemma:ProductOfRelativelyCompact}, \(X^{n+1}\) is \(\sigma\)-relatively compact
    for each \(n \in \omega\).
    So let \(\{ A_{n,k} : k \in \omega \}\) witness the \(\sigma\)-relative compactness of \(X^{n+1}\).
    Given \(\mathscr U \in \Omega_X\) and \(j \leq n\), let \(\sigma_j(\mathscr U,n) \in [\mathscr U]^{<\omega}\)
    be such that \[\bigcup_{\ell=0}^n A_{j,\ell} \subseteq \bigcup \{ U^{j+1} : U \in \sigma_j(\mathscr U,n) \}.\]
    Then, let \[\sigma(\mathscr U,n) = \bigcup_{j=0}^n \sigma_j(\mathscr U, n).\]
    Observe that \(\sigma(\mathscr U, n) \in \left[ \mathscr U \right]^{<\omega}\).
    This completes the definition of \(\sigma\).
    To see that \(\sigma\) is winning, let \(\langle \mathscr U_n : n \in \omega\rangle\)
    be a sequence of \(\omega\)-covers of \(X\) and consider \(\{x_0,\ldots,x_n\} \subseteq X\).
    Note that \(\langle x_0, \ldots, x_n \rangle \in X^{n+1}\).
    Then, there is some \(k \in \omega\) such that
    \[\langle x_0, \ldots , x_n \rangle \in A_{n,k}.\]
    Let \(N = \max \{n,k\}\) and notice that
    \[\langle x_0, \ldots , x_n \rangle \in A_{n,k} \subseteq \bigcup_{j=0}^N A_{n,j}
    \subseteq \bigcup \{ U^{n+1} : U \in \sigma_n(\mathscr U_N,N)\}.\]
    Hence, there must be some \(U \in \sigma_n(\mathscr U_N,N)\) such that \(\langle x_0, \ldots, x_n \rangle \in U^{n+1}\).
    Observe that \[\{x_0,\ldots,x_n\} \subseteq U \in \sigma(\mathscr U_N,N).\]

    Now for \ref{thm:MarkovOmegaRothberger}.
    Note that the equivalence of being Markov Rothberger and topologically countable follows
    immediately from Lemma \ref{lem:NearlyCofinal}.
    The equivalence of being Markov \(\omega\)-Rothberger and topologically countable also
    follows from Lemma \ref{lem:NearlyCofinal} since the lemma asserts that being Markov \(\omega\)-Rothberger
    is equivalent to the condition that \(\widehat{\mathrm{cof}}_X([X]^{<\omega},[X]^{<\omega}) \leq \omega\).
    We need only verify this last condition is equivalent to being topologically countable.

    Suppose \(X\) is topologically countable and let \(\{x_n:n\in\omega\}\) be the witnessing set.
    Note that \(F_n := \{ x_k : k \leq n \}\) constructs a countable set of finite subsets of \(X\).
    Consider any other \(F \in [X]^{<\omega}\).
    For each \(y \in F\), let \(n_y \in \omega\) be such that \(y \in \bigcap \mathcal N_{x_{n_y}}\).
    Let \(M = \max \{ n_y : y \in F \}\) and observe that
    \[F \subseteq \bigcap \mathcal N_{F_M}.\]
    Hence, \(\widehat{\mathrm{cof}}_X([X]^{<\omega},[X]^{<\omega}) \leq \omega\).

    Finally, assume that \(\widehat{\mathrm{cof}}_X([X]^{<\omega},[X]^{<\omega}) \leq \omega\)
    and let \(\{ F_n : n \in \omega \}\) be the witnessing set of finite subsets of \(X\).
    Now, since \(\bigcup_{n\in\omega} F_n\) is countable, we can set \(\{ x_n : n \in \omega \} = \bigcup_{n\in\omega} F_n\).
    To see that this set satisfies the definition of topological countability, let \(x \in X\) and note
    that \(\{x\}\) is a finite subset of \(X\).
    Then there is some \(n \in \omega\) such that \(\{x\} \subseteq \bigcap \mathcal N_{F_n}\).
    To see that there must be some \(y \in F_n\) such that \(x \in \bigcap \mathcal N_y\),
    suppose \(z \in X\) is such that, for each \(y \in F_n\), \(z \not\in \bigcap \mathcal N_y\).
    For each \(y\in F_n\), let \(U_y \in \mathcal N_y\) be such that \(z \not\in U_y\).
    Note then that \(U := \bigcup_{y\in F_n} U_y \in \mathcal N_{F_n}\) and that
    \(z \not\in U\).
    Hence, \(z \not\in \bigcap \mathcal N_{F_n}\).
    It follows that there must be some \(y\in F_n\) such that \(x \in \bigcap \mathcal N_y\).
    Thus, \(X\) is topologically countable.
\end{proof}
\begin{corollary} \label{cor:BasicMarkovOmega}
    Let \(X\) be a space.
    \begin{enumerate}
        \item
        If \(X\) is \(T_1\), then the following are equivalent:
        \begin{enumerate}
            \item 
            \(X\) is Markov \(\omega\)-Rothberger.
            \item
            \(X\) is Markov Rothberger.
            \item
            \(X\) is countable.
        \end{enumerate}
        \item 
        If \(X\) is regular, then the following are equivalent:
        \begin{enumerate}
            \item 
            \(X\) is Markov \(\omega\)-Menger.
            \item
            \(X\) is Markov Menger.
            \item
            \(X\) is \(\sigma\)-compact.
        \end{enumerate}
    \end{enumerate}
\end{corollary}
\begin{proof}
    This follows immediately from Theorem \ref{thm:MarkovOmega},
    Lemma \ref{lem:BasicT1}, and Corollary \ref{cor:RegularSigmaRelCompac}.
\end{proof}

\begin{corollary}
    Let \(\mathcal P\) be any of the properties Markov Menger, Markov \(\omega\)-Menger, Markov Rothberger,
    or Markov \(\omega\)-Rothberger.
    If \(X\) and \(Y\) are both \(\mathcal P\), then \(X \times Y\) is, too.
\end{corollary}
\begin{proof}
    By Theorem \ref{thm:MarkovOmega}, we have only two cases to consider.

    First, suppose both \(X\) and \(Y\) are topologically countable.
    Let \(\{x_n : n \in \omega\}\) and \(\{ y_n : n \in \omega \}\) be the witnessing sets for topological countability
    for \(X\) and \(Y\), respectively.
    Note that \(\{ (x_n,y_m) : n,m \in \omega \}\) is a countable subset of \(X\times Y\).
    For any \((x,y) \in X \times Y\), we can let \(j \in \omega\) and \(k \in \omega\) be such that
    \(x \in \bigcap \mathcal N_{x_j}\) and \(y \in \bigcap \mathcal N_{y_k}\).
    Now, consider any open neighborhood \(W\) of \((x_j,y_k)\).
    There are open sets \(U\) and \(V\) of \(X\) and \(Y\), respectively,
    with \(x_j \in U\), \(y_k \in V\), and \(U \times V \subseteq W\).
    Note then that \((x,y) \in U \times V \subseteq W\).
    Since \(W\) was arbitrary, \((x,y) \in \bigcap \mathcal N_{(x_j,y_k)}\).
    
    Secondly, assume both \(X\) and \(Y\) are \(\sigma\)-relatively compact.
    Let \(\{ A_n : n \in \omega \}\) and \(\{ B_n : n \in \omega \}\) be the witnessing families
    of relatively compact subsets of \(X\) and \(Y\), respectively.
    Note that \(\{ A_n \times B_m : n, m \in \omega \}\) is a countable set of relatively compact subsets
    of \(X\times Y\) by Lemma \ref{lemma:ProductOfRelativelyCompact}.
    Now, for any \((x,y) \in X \times Y\), we can let \(j \in \omega\) and \(k \in \omega\) be such that
    \(x \in A_j\) and \(y \in B_k\); hence, \((x,y) \in A_j \times B_k\).
\end{proof}

\subsection{Considering compact subsets} \label{subsection:CompactCovers}
Throughout the rest of the paper, we'll use \(\mathsf K(X)\) and \(\mathsf K_{\mathrm{rel}}(X)\) to denote the collections
of non-empty compact and non-empty relatively compact subsets of \(X\), respectively.

We start this section by addressing the preservation of various \(k\)-variant properties under finite powers.
To this end, we provide the following, as an easy application of The Wallace Theorem, which will allow us to restrict our attention
to basic covers when dealing with \(k\)-covers of product spaces.
\begin{lemma} \label{lem:BasicKCover}
    If \(\mathscr W\) is a \(k\)-cover of \(X \times Y\), then there is a \(k\)-cover of \(X\times Y\)
    consisting of rectangles \(U \times V\) that refines \(\mathscr W\).
\end{lemma}
\begin{proof}
    Suppose \(\mathscr W\) is a \(k\)-cover of \(X \times Y\).
    For each \(K \in \mathsf K(X \times Y)\), let \(W_K \in \mathscr W\)
    be such that \(\pi_X[K] \times \pi_Y[K] \subseteq W_K\).
    By The Wallace Theorem \cite[Thm. 3.2.10]{Engelking}, there are open sets \(U_K \subseteq X\)
    and \(V_K \subseteq Y\) such that
    \[\pi_X[K] \times \pi_Y[K] \subseteq U_K \times V_K \subseteq W_K.\]
    Hence, \(\{ U_K \times V_K : K \in \mathsf K(X \times Y) \}\) is a \(k\)-cover
    of \(X \times Y\) which consists of rectangles and refines \(\mathscr W\).
\end{proof}
In the case of finite powers, we can refine \(k\)-covers of the finite product space with cubes.
\begin{lemma}[{\cite[Lemma 3]{AppskcoversI}}] \label{lem:kCubeRefinement}
    If \(\mathscr W\) is a \(k\)-cover of \(X^n\) for a positive integer \(n\),
    then there is a \(k\)-cover \(\mathscr U\) of \(X\) with the property that
    \(\{ U^n : U \in \mathscr U \}\) is a \(k\)-cover of \(X^n\) which refines \(\mathscr W\).
\end{lemma}

\begin{theorem} \label{thm:kLindProd}
    Every finite power of a \(k\)-Lindel\"{o}f space is \(k\)-Lindel\"{o}f.
\end{theorem}
\begin{proof}
    Let \(n\) be a positive integer and suppose \(\mathscr W\) is a \(k\)-cover of \(X^n\).
    By Lemma \ref{lem:kCubeRefinement}, we can let \(\mathscr U \in \mathcal K_X\) be such that
    \(\{U^n : U \in \mathscr U\}\) is a \(k\)-cover of \(X^n\) which refines \(\mathscr W\).
    Since \(X\) is assumed to be \(k\)-Lindel\"{o}f, we can choose \(\{U_k : k \in \omega\} \subseteq \mathscr U\)
    such that \(\{U_k:k\in\omega\} \in \mathcal K_X\).
    For each \(k\in\omega\), we can choose \(W_k \in \mathscr W\) such that \(U_k^n \subseteq W_k\).
    To finish the proof, we verify that \(\{ W_k : k \in \omega \}\) is a \(k\)-cover of \(X^n\).
    So let \(K \subseteq X^n\) be compact and note that \(L = \bigcup_{j=1}^n \pi_j[K]\) is a compact subset of \(X\).
    Then there must be \(k \in \omega\) such that \(L \subseteq U_k\).
    It follows that \(K \subseteq L^n \subseteq U_k^n \subseteq W_k\).
\end{proof}

The proof of Theorem \ref{thm:kLindProd} presented here is identical in spirit to those of the following results
from \cite{AppskcoversI}.
\begin{theorem} \label{thm:kProductive}
    Let \(X\) be a space.
    \begin{enumerate}[label=(\arabic*),ref=(\arabic*)]
        \item \label{kMengerProductive}
        If \(X\) is \(k\)-Menger, then every finite power of \(X\) is, too (\cite[Thm. 6]{AppskcoversI}).
        \item \label{kRothbergerProductive}
        If \(X\) is \(k\)-Rothberger, then every finite power of \(X\) is, too (\cite[Thm. 5]{AppskcoversI}).
    \end{enumerate}
\end{theorem}

Before we prove the analogous results to Theorem \ref{thm:DiasScheepers} for \(k\)-covers, we establish a combinatorial lemma that will be useful.
The generality of Lemma \ref{lem:GeneralComboTool} is such that it may be adapted to
Lemma \ref{comboTool} for use in the finite-selection context.

Throughout, we'll use \(\concat\) for the ordered concatenation of words.
If \(s\) is a word and \(x\) is an element of the alphabet, we'll use
\(s \concat x\) in place of \(s \concat \langle x \rangle\), when there is little doubt of ambiguity.

\begin{lemma} \label{lem:GeneralComboTool}
    Suppose \(s : \omega \to \omega^{<\omega}\) is an injection where \(s_0 = \langle \rangle\).
    Suppose further that we have a function \(r : \{ s_n : n \in \omega \} \to [\omega]^{<\omega}\) such that
    \(r(s_n) \subseteq \mathrm{range}(s_n)\) and, for each \(m \geq 1\),
    \(\{ n \in \omega : r(s_n) \subseteq m \}\) is infinite.
    Then there is a bijection \(\beta : \omega^2 \to \omega\) such that, for each \(n \in \omega\),
    \(\langle \beta(n,k) : k \in \omega \rangle\) is strictly increasing
    and \(r(s_n) \subseteq \beta(n,0)\).
\end{lemma}

\begin{proof}
    For \(n \in \mathbb Z^+\), let \(p_n\) be the \(n^{\mathrm{th}}\) prime and set \(\beta^\ast(n,k) = p_n^k\).
    Then let \(\langle \beta^\ast(0,k) : k \in \omega \rangle\) be the increasing enumeration
    of \[\omega \setminus \{ p^n : p \text{ is prime and } n \in \mathbb Z^+ \}.\]
    Note that, for each \(n \in \omega\), \(\langle \beta^\ast(n,k) : k \in \omega \rangle\) is strictly increasing.
    Moreover, the sequence of initial points, \(\langle \beta^\ast(n,0) : n \in \omega \rangle\), is also strictly increasing.

    Define \(\gamma : \omega \to \omega\) recursively as follows.
    Suppose, for \(n \in \omega\), that \(\langle \gamma_k : k < n \rangle\) is defined.
    Let
    \[\gamma_n = \min \{ \lambda \in \omega \setminus \{ \gamma_k : k < n \} : r(s_n) \subseteq \beta^\ast(\lambda, 0) \}.\]
    
    The claim is that \(n \mapsto \gamma_n\), \(\omega\to\omega\), is a bijection.
    \(\gamma\) is clearly injective.
    To show surjectivity, first note that \(\gamma_0 = 0\).
    We proceed by induction.
    Let \(m \geq 1\) and suppose that \(\{ k \in \omega : k < m \} \subseteq \mathrm{range}(\gamma)\).
    Then let
    \[M = \min\{ \lambda \in \omega : \{ k \in \omega : k < m \} \subseteq \{ \gamma_j : j < \lambda \} \}.\]
    
    If \(m \in \{ \gamma_j : j < M \}\), we have nothing to show.
    Otherwise, we can consider
    \[A = \{ \lambda \in \omega : \lambda \geq M \wedge r(s_\lambda) \subseteq \beta^\ast(m,0) \}.\]
    Since \(m \geq 1\), \(\beta^\ast(m,0) > 0\).
    Hence, \(A \neq \emptyset\) since \(\{ \lambda \in \omega : r(s_\lambda) \subseteq \beta^\ast(m,0) \}\)
    is infinite.
    Then we can let \(\ell = \min A\).
    
    If \(m \in \{ \gamma_j : j < \ell \}\), we are done.
    Otherwise, we claim that \(\gamma_\ell = m\).
    By definition,
    \[\gamma_\ell = \min \{ \lambda \in \omega \setminus \{ \gamma_j : j < \ell \} : r(s_\ell) \subseteq \beta^\ast(\lambda, 0) \}.\]
    Since \(\ell \in A\), we know that \(r(s_\ell) \subseteq \beta^\ast(m,0)\).
    Moreover, since \(m \not\in \{ \gamma_j : j < \ell \}\), we see that
    \[m \in \{ \lambda \in \omega \setminus \{ \gamma_j : j < \ell \} : r(s_\ell) \subseteq \beta^\ast(\lambda, 0) \}.\]
    So \(\gamma_\ell \leq m\).
    
    Now, for \(k < m\), notice that there exists \(j < M\) so that \(\gamma_j = k\) by the inductive hypothesis.
    Since \(\gamma\) is injective, \(\gamma_\ell \neq k\).
    Since this is true for any \(k < m\), we have that \(\gamma_\ell = m\).

    Finally, define \(\beta : \omega^2 \to \omega\) by the rule \(\beta(n,k) = \beta^\ast(\gamma_n,k)\).
    This completes the proof.
\end{proof}

The proof of the finite productivity of the strategically \(k\)-Rothberger property
can be seen as an exercise toward the proof of the finite-selection version.
The idea of the proof is identical in spirit to the analogous result in \cite{DiasScheepers};
the reason it is not a direct application of the results of that paper is that \(\mathbb K(X \times Y)\)
and \(\mathbb K(X) \times \mathbb K(Y)\) are, in general, topologically distinct objects.
Recall that \(\mathbb K(X)\) denotes the set \(\mathsf K(X)\) endowed with the Vietoris topology.
See \cite{MichaelSubsets} for more on this topological space.
\begin{theorem} \label{StratKRothProductive}
    If \(X\) and \(Y\) are strategically \(k\)-Rothberger, then \(X \times Y\) is strategically \(k\)-Rothberger.
    In other words, the property of being strategically \(k\)-Rothberger is finitely productive.
\end{theorem}
\begin{proof}
    Let \(\langle s_n : n \in \omega \rangle\) be an enumeration of \(\omega^{<\omega}\) where \(s_0 = \langle \rangle\).
    By Lemma \ref{lem:GeneralComboTool}, we can let
    \(\beta : \omega^2 \to \omega\) be a bijection such that \(\mathrm{range}(s_n) \subseteq \beta(n,0)\)
    for each \(n \in \omega\).
    Let \(\sigma_X\) and \(\sigma_Y\) be winning strategies for P2 in the \(k\)-Rothberger game in \(X\) and \(Y\), respectively.
    Without loss of generality, we consider only basic \(k\)-covers of \(X \times Y\) by Lemma \ref{lem:BasicKCover}.
    In particular, let \[\mathscr T_X \otimes \mathscr T_Y = \{ U \times V : \langle U , V \rangle \in \mathscr T_X \times \mathscr T_Y \}\]
    and \[\mathbb{BK} = \left\{ \mathscr W \subseteq \mathscr T_X \otimes \mathscr T_Y : \mathscr W \in \mathcal K_{X \times Y} \right\}.\]

    We define a winning strategy \(\sigma\) for P2 in the \(k\)-Rothberger game on \(X \times Y\) recursively as follows.
    Let \(n \in\omega \) be given and suppose we have \(\langle \mathscr W_p : p < n \rangle\),
    \(\langle \mathscr W_p^\prime : p < n \rangle\),
    \(\langle \mathscr V_p : p < n \rangle\), \(\langle W_p : p < n \rangle\),
    \(\langle \mathscr U_p : p < n \rangle\), and
    \(\langle K_p : p < n \rangle\) defined.
    Also set \(\langle j, k \rangle \in \omega^2\) to be so that \(\beta(j , k) = n\).
    Note that \[\mathrm{range}(s_j) \subseteq \beta(j,0) \subseteq \beta(j,k) = n.\]

    Given \(\mathscr W_n \in \mathbb{BK}\),  define \(\mathscr W_n^\prime : \mathsf K(X) \to \wp(\mathscr W_n)\)
    by \(\mathscr W_n^\prime(K) = \{ W \in \mathscr W_n : K \subseteq \pi_X[W] \}\).
    Then, define \(\mathscr V_n : \mathsf K(X) \to \mathcal K_Y\) by
    \(\mathscr V_n(K) = \pi_Y[\mathscr W_n^\prime(K) ] = \{ \pi_Y[W] : W \in \mathscr W_n^\prime(K) \}\).
    Let \(W_n(K) \in \mathscr W_n^\prime(K)\) be such that
    \[\pi_Y[W_n(K)] = \sigma_Y\left( \left\langle \mathscr V_{s_j(p)}(K_{s_j(p)}) : p \in \mathrm{dom}(s_j) \right\rangle \concat \mathscr V_n(K) \right).\]
    Note that \(\mathscr U_n := \left\{ \pi_X[W_n(K)] : K \in \mathsf K(X) \right\} \in \mathcal K_X\)
    so we can let \(K_n \in \mathsf K(X)\) be such that
    \[\pi_X[W_n(K_n)] = \sigma_X\left( \left\langle \mathscr U_{\beta(j,p)} : p\leq k \right\rangle \right).\]
    We define
    \[\sigma\left( \left\langle \mathscr W_p : p \leq n \right\rangle \right) = W_n(K_n).\]

    The final thing to show is that \(\sigma\) is winning.
    So let \(A \in \mathsf K(X \times Y)\) be arbitrary and then let
    \(K = \pi_X[A]\) and \(L = \pi_Y[A]\).
    For \(n \in \omega\), suppose we have \(\langle \ell_p : p < n \rangle\), \(\langle M_p : p < n \rangle\),
    and \(\langle m_p : p < n \rangle\) defined.
    Let \(m_n \in \omega\) be so that \(s_{m_n} = \langle \ell_p : p < n \rangle\).
    Observe that
    \[\left\langle \sigma_X \left( \left \langle \mathscr U_{\beta(m_n,p)} : p \leq N \right\rangle \right) : N \in \omega \right\rangle\]
    corresponds to a play of the \(k\)-Rothberger game on \(X\) according to \(\sigma_X\).
    Since \(\sigma_X\) is winning, there is some \(M_n \in \omega\) so that
    \[K \subseteq \sigma_X \left( \left \langle \mathscr U_{\beta(m_n,p)} : p \leq M_n \right\rangle \right).\]
    Let \(\ell_n = \beta(m_n,M_n)\).

    This defines sequences \(\langle \ell_n : n \in \omega \rangle\), \(\langle m_n : n \in \omega \rangle\),
    and \(\langle M_n : n \in \omega \rangle\).
    Observe that
    \[\left\langle \sigma_Y\left( \left\langle \mathscr V_{\ell_p}(K_{\ell_p}) : p \leq N \right\rangle \right) : N \in \omega \right\rangle\]
    corresponds to a run of the \(k\)-Rothberger game on \(Y\) corresponding to \(\sigma_Y\).
    Since \(\sigma_Y\) is winning, there is some \(w\in \omega\) so that
    \[L \subseteq \sigma_Y\left( \left\langle \mathscr V_{\ell_p}(K_{\ell_p}) : p \leq w \right\rangle \right).\]
    Behold that, since \(s_{m_w} = \langle \ell_p : p < w \rangle\),
    \begin{align*}
        L &\subseteq
        \sigma_Y\left( \left\langle \mathscr V_{\ell_p}(K_{\ell_p}) : p \leq w \right\rangle \right)\\
        &= \sigma_Y\left( \left\langle \mathscr V_{\ell_p}(K_{\ell_p}) : p < w \right\rangle \concat \mathscr V_{\ell_w}(K_{\ell_w}) \right)\\
        &= \sigma_Y\left( \left\langle \mathscr V_{s_{m_w}(p)}(K_{s_{m_w}(p)}) : p \in \mathrm{dom}(s_{m_w}) \right\rangle \concat \mathscr V_{\beta(m_w,M_w)}(K_{\beta(m_w,M_w)}) \right)\\
        &= \pi_Y\left[ W_{\beta(m_w,M_w)}(K_{\beta(m_w,M_w)}) \right].
    \end{align*}
    By construction,
    \[K \subseteq \sigma_X\left( \left\langle \mathscr U_{\beta(m_w,p)} : p \leq M_w \right\rangle \right)
    = \pi_X\left[ W_{\beta(m_w,M_w)}(K_{\beta(m_w,M_w)}) \right].\]
    Therefore,
    \[A \subseteq K \times L \subseteq W_{\beta(m_w,M_w)}(K_{\beta(m_w,M_w)})
    = \sigma\left( \left\langle \mathscr W_p : p \leq \beta(m_w,M_w) \right\rangle \right),\]
    finishing the proof.
\end{proof}

\begin{lemma} \label{comboTool}
    Suppose we have an enumeration \(\langle s_n : n \in \omega \rangle\) of \(\bigcup_{n\in\omega} \omega^n \times \omega^n\)
    where \(s_0 = \langle \langle \rangle , \langle \rangle \rangle\).
    For each \(n \in \omega\), let \(s_n^- , s_n^+ \in \omega^{\mathrm{len}(s_n)/2}\) be such that \(s_n = s_n^- \concat s_n^+\).
    Then there is a bijection \(\beta : \omega^2 \to \omega\) such that, for each \(n \in \omega\),
    \(\langle \beta(n,k) : k \in \omega \rangle\) is strictly increasing
    and \(\mathrm{range}(s_n^-) \subseteq \beta(n,0)\).
\end{lemma}
\begin{proof}
    Note that \(n \mapsto s_n\), \(\omega \to \omega^{<\omega}\), is an injection.
    Then let \(r : \{ s_n : n \in \omega \} \to [\omega]^{<\omega}\) be defined by
    \(r(s_n) = \mathrm{range}(s_n^-)\).
    For \(m \geq 1\), note that \(\{ n \in \omega : r(s_n) \subseteq m \}\) is infinite
    since there are arbitrarily long sequences of \(0\)s.
    Hence, Lemma \ref{lem:GeneralComboTool} applies.
\end{proof}

\begin{theorem} \label{thm:StraKMengerProductive}
    If \(X\) and \(Y\) are strategically \(k\)-Menger, then \(X \times Y\) is strategically \(k\)-Menger.
    In other words, the property of being strategically \(k\)-Menger is finitely productive.
\end{theorem}
\begin{proof}
    Let \(\sigma_X\) and \(\sigma_Y\) be winning strategies for P2 in the \(k\)-Menger game in \(X\) and \(Y\), respectively.
    Without loss of generality, we consider only basic \(k\)-covers of \(X \times Y\) by Lemma \ref{lem:BasicKCover}.
    So let \(\mathbb{BK}\) be defined as in the proof of Theorem \ref{StratKRothProductive}.

    We will recursively define a strategy \(\sigma\) for P2 in the \(k\)-Menger game on \(X \times Y\).
    First, fix a choice function \(\vec{\cdot} : [\mathsf K(X)]^{<\omega} \to \mathsf K(X)^\omega\) to be such that
    \(\mathbf K = \mathrm{range}\left( \vec{\mathbf K} \right)\).
    Also, let \(\langle s_n : n \in \omega \rangle\) and \(\beta : \omega^2 \to \omega\) be as in Lemma \ref{comboTool}.

    Now, let \(n \in\omega \) be given and suppose we have \(\langle \mathscr W_p : p < n \rangle\),
    \(\langle \mathscr W_p^\prime : p < n \rangle\),
    \(\langle \mathscr V_p : p < n \rangle\), \(\langle \mathscr G_p : p < n \rangle\),
    \(\langle \mathscr U_p : p < n \rangle\), and
    \(\langle \mathbf F_p : p < n \rangle\) defined.
    Also set \(\langle j, k \rangle \in \omega^2\) such that \(\beta(j , k) = n\).
    Note that \[\mathrm{range}(s^-_j) \subseteq \beta(j,0) \subseteq \beta(j,k) = n.\]

    Given \(\mathscr W_n \in \mathbb{BK}\), define \(\mathscr W_n^\prime : \mathsf K(X) \to \wp(\mathscr W_n)\)
    by \(\mathscr W_n^\prime(K) = \{ W \in \mathscr W_n : K \subseteq \pi_X[W] \}\).
    Then, define \(\mathscr V_n : \mathsf K(X) \to \mathcal K_Y\) by
    \(\mathscr V_n(K) = \pi_Y[\mathscr W_n^\prime(K) ]\).

    We let \(\mathscr G_n(K) \in \left[ \mathscr W_n^\prime(K) \right]^{<\omega}\) be so that
    \[
        \pi_Y[\mathscr G_n(K)] =
        \sigma_Y
        \left(
                \left\langle
                    \mathscr V_{s_j^-(p)} \left( \vec{\mathbf F}_{s_j^-(p)}\left(s_j^+(p)\right) \right) : p \in \mathrm{dom}(s_j^-)
                \right\rangle
            \concat \mathscr V_n(K)
        \right).
    \]
    Then define \[\mathscr U_n = \left\{ \bigcap \pi_X[\mathscr G_n(K)] : K \in \mathsf K(X) \right\} \in \mathcal K_X\]
    and let \(\mathbf F_n \in [\mathsf K(X)]^{<\omega}\) to be such that
    \[\sigma_X\left( \left \langle \mathscr U_{\beta(j,p)} : p \leq k \right\rangle \right) = \left\{ \bigcap \pi_X[\mathscr G_n(K)] : K \in \mathbf F_n \right\}.\]

    Finally, we define
    \[\sigma\left( \left\langle \mathscr W_p : p \leq n \right\rangle \right) = \bigcup_{K \in \mathbf F_n} \mathscr G_n(K).\]

    To finish the proof, we need to show that \(\sigma\) is a winning strategy.
    So let \(A \in \mathsf K(X \times Y)\) be arbitrary and let
    \(K = \pi_X[A]\) and \(L = \pi_Y[A]\).
    For \(n \in \omega\), suppose we have
    \(\langle m_p : p < n \rangle\), \(\langle M_p : p < n \rangle\), \(\vec{\ell} = \langle \ell_p : p < n \rangle\),
    \(\vec{u} = \langle u_p : p < n \rangle\), and \(\langle K_p : p < n \rangle\) defined.
    Then we can let \(m_n \in \omega\) be so that \(s_{m_n} = \langle \vec{\ell}, \vec{u} \rangle\).
    Observe that
    \[\left\langle \sigma_X \left( \left \langle \mathscr U_{\beta(m_n,p)} : p \leq N \right\rangle \right) : N \in \omega \right\rangle\]
    corresponds to a play of the \(k\)-Menger game on \(X\) according to \(\sigma_X\).
    Since \(\sigma_X\) is winning, there is some \(M_n \in \omega\)
    and \(U \in \sigma_X\left( \left \langle \mathscr U_{\beta(m_n,p)} : p \leq M_n \right\rangle \right)\) such that \(K \subseteq U\).
    It follows that there is some \(u_n \in \omega\) such that
    \[K \subseteq \bigcap \pi_X\left[ \mathscr G_{\beta(m_n,M_n)}\left( \vec{\mathbf F}_{\beta(m_n,M_n)}\left( u_n \right) \right) \right].\]
    Let \(\ell_n = \beta(m_n,M_n)\) and \(K_n = \vec{\mathbf F}_{\beta(m_n,M_n)}\left( u_n \right)\).
    Note that \(K_n \in \mathbf F_{\beta(m_n,M_n)}\).

    This defines sequences \(\langle m_n : n \in \omega \rangle\), \(\langle M_n : n \in \omega\rangle\),
    \(\langle \ell_n : n \in \omega \rangle\), \(\langle u_n : n \in \omega \rangle\), and \(\langle K_n : n \in \omega \rangle\).
    Note that
    \[
        \left\langle
            \sigma_Y\left(
                \left\langle
                    \mathscr V_{\ell_p}\left( K_p \right) : p \leq N
                \right\rangle
            \right) : N \in \omega
        \right\rangle
    \]
    corresponds to a play of the \(k\)-Menger game on \(Y\) according to \(\sigma_Y\).
    Since \(\sigma_Y\) is winning, there is some \(w \in \omega\) and \(V \in \sigma_Y\left( \left\langle \mathscr V_{\ell_p}\left( x_p \right) : p \leq w \right\rangle \right)\)
    such that \(L \subseteq V\).
    
    Note that \(s_{m_w} = \left\langle \langle \ell_p : p < w \rangle , \langle u_p : p < w \rangle \right\rangle\).
    Then
    \begin{align*}
        \pi_Y\left[ \mathscr G_{\ell_w}(K_w) \right]
        &=\pi_Y\left[\mathscr G_{\beta(m_w,M_w)}\left(\vec{\mathbf F}_{\beta(m_w,M_w)}(u_w)\right)\right]\\
        &= \sigma_Y
        \left(
                \left\langle
                    \mathscr V_{s_{m_w}^-(p)} \left( \vec{\mathbf F}_{s_{m_w}^-(p)}\left(s_{m_w}^+(p)\right) \right)
                    : p \in \mathrm{dom}(s_{m_w}^-)
                \right\rangle
            \concat \mathscr V_{\beta(m_w,M_w)}(K_w)
        \right)\\
        &= \sigma_Y
        \left(
                \left\langle
                    \mathscr V_{\ell_p} \left( \vec{\mathbf F}_{\ell_p}\left(u_p\right) \right) : p < w
                \right\rangle
            \concat \mathscr V_{\beta(m_w,M_w)}\left(\vec{\mathbf F}_{\beta(m_w,M_w)}\left( u_w \right)\right)
        \right)\\
        &= \sigma_Y \left(
                \left\langle
                    \mathscr V_{\ell_p} \left( \vec{\mathbf F}_{\ell_p}\left(u_p\right) \right) : p \leq w
                \right\rangle
            \right).
    \end{align*}
    Hence, there is some \(V \in \pi_Y \left[\mathscr G_{\ell_w}\left(K_w\right)\right]\) such that \(L \subseteq V\).
    Thus, there is some \[W \in \mathscr G_{\ell_w}\left(K_w\right)\] such that \(L \subseteq \pi_Y[W]\).
    Observe that, since \(K_w \in \mathbf F_{\ell_w}\),
    \[W \in \bigcup_{u \in \mathbf F_{\ell_w}} \mathscr G_{\ell_w}(u)
    = \sigma\left( \left\langle \mathscr W_p : p \leq \ell_w \right\rangle \right).\]
    Finally, note that
    \[K \subseteq \bigcap \pi_X \left[ \mathscr G_{\ell_w}\left(K_w\right)\right] \subseteq \pi_X[W].\]
    Therefore, \(A \subseteq K \times L \subseteq W\).
\end{proof}

When it comes to \(k\)-covers, hemicompactness plays an important role.
\begin{proposition}[{\cite[Prop. 5]{AppskcoversII}}] \label{prop:T1FirstCountableHemicompact}
    For any \(T_1\) first-countable space, the following are equivalent:
    \begin{enumerate}[label=(\alph*)]
        \item
        \(X\) is hemicompact.
        \item
        \(X \models \mathsf S_{\mathrm{fin}}(\mathcal K, \mathcal K)\).
        \item
        \(X \models \mathsf S_1(\mathcal K, \mathcal K)\).
    \end{enumerate}
\end{proposition}
The reader would observe that the \(T_1\) assumption is expressly used in the proof of \cite[Prop. 5]{AppskcoversII}
though it is not explicitly mentioned in the hypotheses.

Even if we relax first-countability, we still get a related characterization of
hemicompactness for \(T_1\) spaces.
\begin{corollary} \label{cor:T1Hemicompact}
    For any \(T_1\) space, the following are equivalent:
    \begin{enumerate}[label=(\alph*)]
        \item
        \(X\) is hemicompact.
        \item
        \(\mathrm{II} \underset{\mathrm{mark}}{\uparrow} \mathsf G_1(\mathcal K_X, \mathcal K_X)\).
    \end{enumerate}
\end{corollary}
\begin{proof}
    This follows immediately from Lemmas \ref{lem:NearlyCofinal} and \ref{lem:BasicT1}.
\end{proof}
Note that Lemma \ref{lem:NearlyCofinal} applies regardless of the \(T_1\) assumption,
so we introduce a descriptive term for the property of being Markov \(k\)-Rothberger.
\begin{definition}
    We say that a space \(X\) is \emph{nearly hemicompact} if
    \(\widehat{\mathrm{cof}}_X(\mathsf K(X), \mathsf K(X)) \leq \omega\).
\end{definition}
\begin{corollary} \label{cor:NearlyHemicompact}
    A space is nearly hemicompact if and only if it is Markov \(k\)-Rothberger.
\end{corollary}
Evidently, every topologically countable space is nearly hemicompact.
On the other hand, every uncountable \(T_1\) hemicompact space is nearly hemicompact but not topologically countable.

To prove Theorem \ref{MarkovKProductive}\ref{MarkovKMengerProd}, we can recycle, as we did for Theorem \ref{thm:StraKMengerProductive},
the idea behind the proof of Theorem \ref{thm:DiasScheepers}\ref{DiasScheepersMenger} of Dias and Scheepers \cite{DiasScheepers},
in which they thank L. Aurichi.
In the Markov case, as one would expect, the combinatorial obstacles are significantly reduced.

We would like to compare the current proofs of parts \ref{MarkovKRothProd} and
\ref{MarkovKMengerProd} in Theorem \ref{MarkovKProductive} in light of the similarity between the proofs
of Theorems \ref{StratKRothProductive} and \ref{thm:StraKMengerProductive}.
As can be seen, having the covering characterization of Corollary \ref{cor:NearlyHemicompact} significantly
simplifies matters.

\begin{theorem} \label{MarkovKProductive}
    Let \(X\) and \(Y\) be spaces.
    \begin{enumerate}[label=(\arabic*),ref=(\arabic*)]
        \item \label{MarkovKRothProd}
        If \(X\) and \(Y\) are both Markov \(k\)-Rothberger, then so is \(X \times Y\).
        In other words, the property of being Markov \(k\)-Rothberger is finitely productive.
        \item \label{MarkovKMengerProd}
        If \(X\) and \(Y\) are both Markov \(k\)-Menger, then so is \(X \times Y\).
        In other words, the property of being Markov \(k\)-Menger is finitely productive.
    \end{enumerate}
\end{theorem}
\begin{proof}
    We start by proving \ref{MarkovKRothProd} by using the characterization of Corollary \ref{cor:NearlyHemicompact}.
    So let \(\{ K_n : n \in \omega \}\) and \(\{ L_n : n \in \omega \}\) be families of compact subsets of
    \(X\) and \(Y\), respectively, that witness that they are nearly hemicompact.
    Note that \(\{ K_n \times L_m : n,m \in \omega \}\) is a countable set of compact subsets of \(X \times Y\).
    Now, let \(E \subseteq X \times Y\) be an arbitrary compact set.
    Note that \(\pi_X[E]\) and \(\pi_Y[E]\) are compact subsets of \(X\) and \(Y\), respectively.
    So we can let \(j \in \omega\) and \(k\in \omega\) be such that \(\pi_X[E] \subseteq \bigcap \mathcal N_{K_j}\)
    and \(\pi_Y[E] \subseteq \bigcap \mathcal N_{L_k}\).

    Now, let \(W\) be an arbitrary open subset of \(X \times Y\) with \(K_j \times L_k \subseteq W\).
    By The Wallace Theorem \cite[Thm. 3.2.10]{Engelking}, there are open subsets \(U\) and \(V\) of \(X\) and \(Y\), respectively,
    such that \(K_j \times L_k \subseteq U \times V \subseteq W\).
    Finally, note that \[E \subseteq \pi_X[E] \times \pi_Y[E] \subseteq U \times V \subseteq W.\]
    That is, \(E \subseteq \bigcap \mathcal N_{K_j \times L_k}\).

    We now prove \ref{MarkovKMengerProd}.
    Let \(\sigma_X\) and \(\sigma_Y\) be Markov strategies for P2 in the \(k\)-Menger games on \(X\) and \(Y\), respectively.
    Let \(\beta : \omega^2 \to \omega\) be a bijection, \(\mathbb{BK}\) be as it was defined in Theorem \ref{StratKRothProductive},
    and \(n = \beta(j,k)\).
    For each \(\mathscr W \in \mathbb{BK}\) and \(K \in \mathsf K(X)\), define
    \[\left.\mathscr W\right|_K = \{ W \in \mathscr W : K \subseteq \pi_X[W] \}.\]
    Fix a choice function \(\gamma : \mathbb{BK} \times \mathsf K(X) \times \omega \to [\mathscr T_{X \times Y}]^{<\omega}\)
    to be such that \(\gamma(\mathscr W, K , k) \in \left[ \left.\mathscr W\right|_K \right]^{<\omega}\)
    for each \((\mathscr W , K) \in \mathbb{BK} \times \mathsf K(X)\)
    and \[\pi_Y[\gamma(\mathscr W, K , k)] = \sigma_Y\left(\pi_Y\left[\left. \mathscr W \right|_K \right],k\right).\]
    Note that \[\mathscr W^{X,k}:= \left\{ \bigcap \pi_X[\gamma(\mathscr W, K , k)] : K \in \mathsf K(X) \right\} \in \mathcal K_X.\]
    Let \(\mathscr F(\mathscr W,j,k) \in [\mathsf K(X)]^{<\omega}\) be such that
    \[\sigma_X(\mathscr W^{X,k},j) = \left\{ \bigcap \pi_X[\gamma(\mathscr W, K , k)]: K \in \mathscr F(\mathscr W,j,k) \right\}.\]
    Then define
    \[\sigma(\mathscr W, n ) = \sigma(\mathscr W, \beta(j,k))
    = \bigcup \left\{ \gamma(\mathscr W, K , k) : K \in \mathscr F(\mathscr W,j,k) \right\}
    \in \left[ \mathscr W \right]^{<\omega}.\]
    Note that \(\sigma\) is a Markov strategy for P2 in the \(k\)-Menger game on \(X \times Y\).
    We will show that it is winning.
    
    Let \(\langle \mathscr W_n : n \in \omega \rangle\) be a sequence of \(\mathbb{BK}\)
    and let \(E \subseteq X \times Y\) be compact.
    For every \(k \in \omega\), \[\bigcup\left\{ \sigma_X\left(\mathscr W_{\beta(j,k)}^{X,k},j\right) : j \in \omega \right\} \in \mathcal K_X.\]
    So we can choose \(j_k \in \omega\) and \(K_{j_k,k} \in \mathscr F(\mathscr W_{\beta(j_k,k)}, j_k , k)\)
    for every \(k \in \omega\) such that
    \[\pi_X[E] \subseteq \bigcap \pi_X\left[ \gamma\left( \mathscr W_{\beta(j_k,k)}, K_{j_k,k} , k \right) \right].\]
    Now, \[\bigcup\left\{ \sigma_Y\left( \pi_Y \left[ \left. \mathscr W_{\beta(j_k,k)} \right|_{K_{j_k,k}} \right],k\right) : k \in \omega \right\} \in \mathcal K_Y,\]
    so we can fix \(m \in \omega\) and \(W \in \gamma\left( \mathscr W_{\beta(j_m,m)}, K_{j_m,m}, m \right)\) such that
    \[\pi_Y[E] \subseteq \pi_Y[W] \in \sigma_Y\left( \pi_Y\left[ \left. \mathscr W_{\beta(j_m,m)} \right|_{K_{j_m,m}} \right], m\right).\]
    Note also that, since \(W \in \gamma\left( \mathscr W_{\beta(j_m,m)}, K_{j_m,m}, m \right)\),
    \[\pi_X[E] \subseteq \bigcap \pi_X\left[ \gamma\left( \mathscr W_{\beta(j_m,m)}, K_{j_m,m} , m \right) \right] \subseteq \pi_X[W].\]
    Hence, \[E \subseteq \pi_X[E] \times \pi_Y[E] \subseteq \pi_X[W] \times \pi_Y[W] = W.\]
    Finally, by construction \(W \in \sigma(\mathscr W_{\beta(j_m,m)}, \beta(j_m,m))\), so \(\sigma\) is winning.
\end{proof}

Inspired by the property equivalent to being Markov \(\omega\)-Menger, the property of being \(\sigma\)-relatively compact,
we introduce modifications to hemicompactness.

\begin{definition}
    A space \(X\) is \emph{relatively hemicompact} if \(\mathrm{cof}(\mathsf K_{\mathrm{rel}}(X),\mathsf K_{\mathrm{rel}}(X)) \leq \omega\);
    in other words, if there exists a countable set \(\{ A_n : n \in \omega \}\) of
    sets that are relatively compact in \(X\) such that, for every relatively compact \(E\) in \(X\),
    there is some \(n \in \omega\) such that \(E \subseteq A_n\).
\end{definition}
\begin{definition}
    A space \(X\) is \emph{weakly relatively hemicompact} if
    \(\mathrm{cof}(\mathsf K_{\mathrm{rel}}(X),\mathsf K(X)) \leq \omega\); in other words,
    if there exists a countable set \(\{ A_n : n \in \omega \}\) of relatively compact subsets of \(X\) such that,
    for each compact \(K \subseteq X\), there exists \(n\in \omega\) with \(K \subseteq A_n\).
\end{definition}
We use the adjective ``weakly'' here since every compact set is relatively compact.
Hence, every relatively hemicompact space is weakly relatively hemicompact, but the converse may not obtain.
\begin{remark} \label{rmk:relHemiImplications}
    In general,
    \begin{center}
        \begin{tikzpicture}
            \node (hemicompact) at (-4.5,1) {hemicompact};
            \node (relHemicompact) at (-4.5,0) {relatively hemicompact};
            \node (weaklyRelHemicompact) at (1,0) {weakly relatively hemicompact};
            \node (sigmaRelCompact) at (6.5,0) {\(\sigma\)-relatively compact};
            \draw [-implies,double equal sign distance] (hemicompact.east) -- (weaklyRelHemicompact);
            \draw [-implies,double equal sign distance] (relHemicompact) -- (weaklyRelHemicompact);
            \draw [-implies,double equal sign distance] (weaklyRelHemicompact) -- (sigmaRelCompact);
        \end{tikzpicture}
    \end{center}
\end{remark}
\begin{lemma} \label{lem:RegularRelHemi}
    In the realm of regular spaces, the properties of being hemicompact, relatively hemicompact, and weakly relatively hemicompact
    are all equivalent.
\end{lemma}
\begin{proof}
    If \(X\) is regular and relatively hemicompact, we start by letting \(\{ A_n : n \in \omega \}\)
    be a sequence of relatively compact subsets witnessing relative hemicompactness for \(X\).
    Since \(X\) is regular, \(\mathrm{cl}_X(A_n)\) is compact for each \(n \in \omega\) by Lemma \ref{lem:RelCompactRegular}.
    To see that \(X\) is hemicompact, consider any \(K \subseteq X\) compact.
    Since \(K\) is also relatively compact, there is some \(n\in\omega\) such that \(K \subseteq A_n \subseteq \mathrm{cl}_X(A_n)\).
    So \(X\) is hemicompact.
    
    Now assume \(X\) is regular and hemicompact.
    Let \(\{ K_n : n \in \omega \}\) be a set of compact subsets witnessing the hemicompactness of \(X\).
    Note that each \(K_n\) is also relatively compact.
    So consider \(A \subseteq X\) which is relatively compact.
    Since \(X\) is regular, \(\mathrm{cl}_X(A)\) is compact, and thus, there is some \(n\in\omega\)
    such that \(A \subseteq \mathrm{cl}_X(A) \subseteq K_n\).
    That is, \(X\) is relatively hemicompact.

    Now that we've shown that the properties of being hemicompact and relatively hemicompact are equivalent
    in the realm of regular spaces, we finish the proof by showing that a regular weakly relatively hemicompact
    space is hemicompact.
    This follows immediately from the fact that the closure of a relatively compact subset of a regular space
    is compact.
\end{proof}

To make some general connections to the property of being Markov \(k\)-Menger,
we introduce a natural modification to the notion of \(k\)-covers relative to the
family of relatively compact subsets of a space.
\begin{definition}
    Let \(\mathcal K_X^{\mathrm{rel}} = \mathcal O_X(\mathsf K_{\mathrm{rel}}(X))\), the collection of all open covers \(\mathscr U\) such that,
    for every relatively compact \(A\) in \(X\), there is some \(U \in \mathscr U\) such that \(A \subseteq U\).
    We will refer to these covers as \emph{relative \(k\)-covers}; we will refer to
    \(\mathsf G_{\mathrm{fin}}(\mathcal K_X^{\mathrm{rel}},\mathcal K_X^{\mathrm{rel}})\) and
    \(\mathsf G_{1}(\mathcal K_X^{\mathrm{rel}},\mathcal K_X^{\mathrm{rel}})\) as the \emph{relative \(k\)-Menger game}
    and the \emph{relative \(k\)-Rothberger game} on \(X\), respectively.
\end{definition}
The following lemma follows immediately from the definitions.
\begin{lemma} \label{lem:OpenCoverClosedUndefFinite}
    If \(\mathscr U\) is an open cover of a space \(X\), then
    \[\mathscr U^{\mathrm{fin}} := \left\{ \bigcup \mathscr F : \mathscr F \in \left[ \mathscr U \right]^{<\omega} \right\}
    \in \mathcal K_X^{\mathrm{rel}}.\]
\end{lemma}

\begin{lemma} \label{lem:BasicMarkovRelCompact}
    For any space \(X\), if \(\varphi : \mathcal K_X^{\mathrm{rel}} \to \left[ \mathscr T_X \right]^{<\omega}\)
    is such that \(\varphi(\mathscr U) \in \left[ \mathscr U \right]^{<\omega}\) for each
    \(\mathscr U \in \mathcal K_X^{\mathrm{rel}}\),
    then \[A := \bigcap \left\{ \bigcup \varphi\left(\mathscr U^{\mathrm{fin}}\right) : \mathscr U \in \mathcal O_X \right\}\]
    is relatively compact in \(X\).
\end{lemma}
\begin{proof}
    By Lemma \ref{lem:OpenCoverClosedUndefFinite}, \(A\) is defined.
    So we need only show that \(A\) is relatively compact in \(X\).
    Indeed, consider any open cover \(\mathscr U\) of \(X\).
    Then \(A \subseteq \bigcup \varphi\left( \mathscr U^{\mathrm{fin}} \right)\).
    Since \(\varphi\left( \mathscr U^{\mathrm{fin}} \right) \in \left[ \mathscr U^{\mathrm{fin}} \right]^{<\omega}\),
    we see that \(A\) is covered by a finite subset of \(\mathscr U\).
\end{proof}

\begin{proposition} \label{prop:WeirdGeneralThing}
    For any space \(X\), the following are equivalent:
    \begin{enumerate}[label=(\alph*),ref=(\alph*)]
        \item \label{cofinalityCondition}
        \(\mathrm{cof}_X(\mathsf K_{\mathrm{rel}}(X),\mathcal B) \leq \omega\).
        \item  \label{MengerCofinality}
        \(\mathrm{II} \underset{\mathrm{mark}}{\uparrow}
        \mathsf G_{\mathrm{fin}}(\mathcal K^{\mathrm{rel}}_X, \mathcal O_X(\mathcal B))\).
        \item \label{RothbergerCofinality}
        \(\mathrm{II} \underset{\mathrm{mark}}{\uparrow} \mathsf G_1(\mathcal K^{\mathrm{rel}}_X, \mathcal O_X(\mathcal B))\).
    \end{enumerate}
\end{proposition}
\begin{proof}
    \ref{cofinalityCondition}\(\implies\)\ref{RothbergerCofinality}:
    Let \(\{ A_n : n \in \omega \} \subseteq \mathsf K_{\mathrm{rel}}(X)\) witness that
    \(\mathrm{cof}_X(\mathsf K_{\mathrm{rel}}(X),\mathcal B) \leq \omega\) and define \(\sigma\) in the following way.
    Given \(\mathscr U \in \mathcal K_X^{\mathrm{rel}}\), let \(\sigma(\mathscr U,n) \in \mathscr U\)
    be such that \(A_n \subseteq \sigma(\mathscr U,n)\).
    Then \(\sigma\) is a winning Markov strategy.
    Indeed, for any sequence \(\langle \mathscr U_n : n \in \omega \rangle\)
    of relative \(k\)-covers of \(X\), if \(B \in \mathcal B\), there is some \(n\in\omega\)
    such that \(B \subseteq A_n \subseteq \sigma(\mathscr U_n , n)\).

    \ref{RothbergerCofinality}\(\implies\)\ref{MengerCofinality} is obvious.

    To prove \ref{MengerCofinality}\(\implies\)\ref{cofinalityCondition}, we proceed by the contrapositive.
    Let \(\sigma\) be a Markov strategy for P2 in
    \(\mathsf G_{\mathrm{fin}}(\mathcal K^{\mathrm{rel}}_X, \mathcal O_X(\mathcal B))\)
    and define
    \[A_n = \bigcap \left\{ \bigcup \sigma\left(\mathscr U^{\mathrm{fin}},n\right) : \mathscr U \in \mathcal O_X \right\}.\]
    By Lemma \ref{lem:BasicMarkovRelCompact}, \(A_n\) is relatively compact for each \(n\in\omega\).
    
    By hypothesis, there is some \(B \in \mathcal B\) that is not covered by any \(A_n\).
    So let \(x_n \in B \setminus A_n\) for each \(n \in \omega\).
    In the \(n^{\mathrm{th}}\) inning, let P1 choose \(\mathscr U_n \in \mathcal O_X\)
    such that \[x_n \not\in \bigcup \sigma \left( \mathscr U_n^{\mathrm{fin}} , n \right).\]
    It follows that \(\langle \mathscr U_n^{\mathrm{fin}} : n \in \omega \rangle\)
    is a play by P1 that beats \(\sigma\).
\end{proof}

\begin{lemma} \label{lem:MarkovKMengerRelHemi}
    If \(X\) is Markov \(k\)-Menger, then \(X\) is weakly relatively hemicompact.
\end{lemma}
\begin{proof}
    First, note that
    \[\mathrm{II} \underset{\mathrm{mark}}{\uparrow} \mathsf G_{\mathrm{fin}}(\mathcal K_X, \mathcal K_X)
    \implies \mathrm{II} \underset{\mathrm{mark}}{\uparrow} \mathsf G_{\mathrm{fin}}(\mathcal K^{\mathrm{rel}}_X, \mathcal K_X)\]
    since \(\mathcal K_X^{\mathrm{rel}} \subseteq \mathcal K_X\) (which follows from the fact that every compact set is relatively compact).
    So then Proposition \ref{prop:WeirdGeneralThing} applies to assert that \(\mathrm{cof}_X(\mathsf K_{\mathrm{rel}}(X),\mathsf K(X)) \leq \omega\);
    that is, \(X\) is weakly relatively hemicompact.
\end{proof}

We now collect some particular applications of Proposition \ref{prop:WeirdGeneralThing}.
\begin{corollary} \label{cor:CorToWeird}
    Let \(X\) be a space.
    \begin{enumerate}[ref=(\arabic*)]
        \item \label{cor:RelHemiChar}
        The following are equivalent:
        \begin{enumerate}[label=(\alph*),ref=(\alph*)]
            \item
            \(X\) is relatively hemicompact.
            \item
            \(\mathrm{II} \underset{\mathrm{mark}}{\uparrow} \mathsf G_{\mathrm{fin}}(\mathcal K^{\mathrm{rel}}_X, \mathcal K^{\mathrm{rel}}_X)\).
            \item
            \(\mathrm{II} \underset{\mathrm{mark}}{\uparrow} \mathsf G_1(\mathcal K^{\mathrm{rel}}_X, \mathcal K^{\mathrm{rel}}_X)\).
        \end{enumerate}
        \item \label{cor:WeaklyRelHemiChar}
        The following are equivalent:
        \begin{enumerate}[label=(\alph*),ref=(\alph*)]
            \item
            \(X\) is weakly relatively hemicompact.
            \item
            \(\mathrm{II} \underset{\mathrm{mark}}{\uparrow} \mathsf G_{\mathrm{fin}}(\mathcal K^{\mathrm{rel}}_X, \mathcal K_X)\).
            \item
            \(\mathrm{II} \underset{\mathrm{mark}}{\uparrow} \mathsf G_1(\mathcal K^{\mathrm{rel}}_X, \mathcal K_X)\).
        \end{enumerate}
        \item \label{cor:SigmaRelChar}
        The following are equivalent:
        \begin{enumerate}[label=(\alph*),ref=(\alph*)]
            \item
            \(X\) is \(\sigma\)-relatively compact.
            \item
            \(\mathrm{II} \underset{\mathrm{mark}}{\uparrow} \mathsf G_{\mathrm{fin}}(\mathcal K^{\mathrm{rel}}_X, \mathcal O_X)\).
            \item
            \(\mathrm{II} \underset{\mathrm{mark}}{\uparrow} \mathsf G_1(\mathcal K^{\mathrm{rel}}_X, \mathcal O_X)\).
        \end{enumerate}
    \end{enumerate}
\end{corollary}

\begin{corollary} \label{cor:HemicompactMarkov}
    For any regular space, the following are equivalent:
    \begin{enumerate}[label=(\alph*)]
        \item
        \(X\) is hemicompact.
        \item
        \(\mathrm{II} \underset{\mathrm{mark}}{\uparrow} \mathsf G_{\mathrm{fin}}(\mathcal K_X, \mathcal K_X)\).
        \item
        \(\mathrm{II} \underset{\mathrm{mark}}{\uparrow} \mathsf G_1(\mathcal K_X, \mathcal K_X)\).
    \end{enumerate}
\end{corollary}
\begin{proof}
    By Corollary \ref{cor:CorToWeird}\ref{cor:RelHemiChar} and Lemma \ref{lem:RegularRelHemi}, it suffices to show that 
    \(\mathcal K_X = \mathcal K_X^{\mathrm{rel}}\) when \(X\) is regular.

    Since every compact set is relatively compact, \(\mathcal K_X^{\mathrm{rel}} \subseteq \mathcal K_X\).
    We show that \(\mathcal K_X \subseteq \mathcal K_X^{\mathrm{rel}}\) under the assumption that \(X\) is regular.
    So let \(\mathscr U \in \mathcal K_X\) and suppose \(A \subseteq X\) is relatively compact.
    By Lemma \ref{lem:RelCompactRegular}, \(\mathrm{cl}_X(A)\) is compact,
    so there is some \(U \in \mathscr U\) such that \(\mathrm{cl}_X(A) \subseteq U\).
    That is, \(\mathscr U \in \mathcal K_X^{\mathrm{rel}}\).
\end{proof}

\subsection{Relating the various games} \label{subsection:RelatingTheGames}

The following proposition appears as \cite[Lemma 3.7]{CHVietoris}, but we offer it here as a consequence of
the investigations of this paper without any separation axiom assumptions;
it also summaries many of the implications in Figure \ref{fig:Omega}.
\begin{theorem}
    In general,
    \[\mathsf{G}_{1}(\Omega, \Omega) \leq^+_{\mathrm{II}} \mathsf G_{1}(\mathcal O, \mathcal O).\]
\end{theorem}
\begin{proof}
    This follows immediately from
    Theorems \ref{thm:GerlitsNagy}, \ref{thm:OmegaSelectionChar}\ref{OmegaSelectionRothbergerChar},
    \ref{thm:StrategicOmega}\ref{StrategicRothChar}, and \ref{thm:MarkovOmega}\ref{thm:MarkovOmegaRothberger}.
\end{proof}

As will be established with Example \ref{example:Reals},
\[\mathsf{G}_1(\mathcal K, \mathcal K) \nleq^+_{\mathrm{II}} \mathsf G_1(\Omega, \Omega),\]
generally.
However, for finite selections, we have the following theorem
which proves the ``\(k\)-Menger \(\implies\) \(\omega\)-Menger'' variations that appear in Figure \ref{fig:KOmega}.
\begin{theorem}\label{thm:KImpliesOmegaImpliesOpen}
    In general,
    \[\mathsf{G}_{\mathrm{fin}}(\mathcal K, \mathcal K)
    \leq^+_{\mathrm{II}} \mathsf{G}_{\mathrm{fin}}(\Omega, \Omega)
    \leq^+_{\mathrm{II}} \mathsf G_{\mathrm{fin}}(\mathcal O, \mathcal O).\]
\end{theorem}
\begin{proof}
    The \[\mathsf{G}_{\mathrm{fin}}(\Omega, \Omega)
    \leq^+_{\mathrm{II}} \mathsf G_{\mathrm{fin}}(\mathcal O, \mathcal O)\]
    portion is \cite[Lemma 3.4]{CHVietoris}, but also follows immediately from Theorems
    \ref{thm:GerlitsNagy}, \ref{thm:OmegaSelectionChar}\ref{OmegaSelectionMengerChar},
    \ref{thm:StrategicOmega}\ref{StrategicMengerChar}, and \ref{thm:MarkovOmega}\ref{thm:MarkovOmegaMenger}.

    Now we show that \[\mathsf{G}_{\mathrm{fin}}(\mathcal K, \mathcal K)
    \leq^+_{\mathrm{II}} \mathsf{G}_{\mathrm{fin}}(\Omega, \Omega).\]

    First, assume that \(X\) is Markov \(k\)-Menger.
    By Lemma \ref{lem:MarkovKMengerRelHemi}, \(X\) is weakly relatively hemicompact which implies that
    \(X\) is \(\sigma\)-relatively compact (Remark \ref{rmk:relHemiImplications}).
    Hence, by Theorem \ref{thm:MarkovOmega}, \(X\) is Markov \(\omega\)-Menger.

    Assume \(X\) is strategically \(k\)-Menger.
    By Theorem \ref{thm:StrategicOmega}, it suffices to show that \(X\) is strategically Menger.
    To see this, just take the closure of open covers under finite unions to create \(k\)-covers.
    When you apply P2's strategies to these \(k\)-covers, P2's selections can be seen as
    finite selections from the original open covers.
    Since P2's selections win in the \(k\)-Menger game, the decomposed collections in the Menger game
    must still cover the space.

    Assume \(X\) is \(k\)-Menger.
    Note that any space which is \(k\)-Menger is Menger.
    This follows from similar reasoning as above by closing open covers under finite unions.
    Then, by Theorem \ref{thm:kProductive}\ref{kMengerProductive},
    every finite power of \(X\) is Menger.
    So, by Theorem \ref{thm:OmegaSelectionChar}\ref{OmegaSelectionMengerChar}, \(X\) is
    \(\omega\)-Menger.

    Lastly, assume \(X\) is \(k\)-Lindel\"{o}f.
    Note that any space which is \(k\)-Lindel\"{o}f is
    Lindel\"{o}f by taking an open cover and closing it under finite unions, just as above.
    Then, by Theorem \ref{thm:kLindProd}, every finite power of \(X\) is Lindel\"{o}f.
    Hence, by Theorem \ref{thm:GerlitsNagy}, \(X\) is \(\omega\)-Lindel\"{o}f.
\end{proof}

\section{Examples} \label{section:Examples}

In this section, we provide some examples that separate some of the properties in Figures \ref{fig:KOmega} and \ref{fig:Omega};
we also answer a question posed in \cite{CHCompactOpen}.

Note that \(\mathsf G_1(\mathcal N(\mathsf K(X)), \neg \mathcal K_X)\)
is a variant of the compact-open game, the game in which P1 chooses a compact set and P2 responds with
an open set containing that compact set,
where P2 is trying to avoid forming a \(k\)-cover of \(X\).

Recall also that the finite-open game on \(X\), which corresponds to \(\mathsf G_1(\mathcal N([X]^{<\omega}) , \neg \mathcal O_X)\) in our notation,
is the game in which P1 chooses a finite set and P2 responds
with an open set containing that finite set, where P2 is trying to avoid forming a cover of \(X\).

In \cite{CHCompactOpen}, it was asked if there is a space \(X\)
for which P1 has a winning strategy in \(\mathsf G_1(\mathcal N(\mathsf K(X)), \neg \mathcal K_X)\),
but P1 doesn't have a winning strategy in the finite-open game on \(X\) and P1 doesn't have a predetermined
winning strategy in \(\mathsf G_1(\mathcal N(\mathsf K(X)), \neg \mathcal K_X)\).
We can provide a straightforward answer in the affirmative.

\begin{example} \label{example:CHQ}
    There is a \(T_5\)
    space for which P1 has a winning strategy in \(\mathsf G_1(\mathcal N(\mathsf K(X)), \neg \mathcal K_X)\),
    but P1 doesn't have a winning strategy in the finite-open game on \(X\) and P1 doesn't have a predetermined
    winning strategy in \(\mathsf G_1(\mathcal N(\mathsf K(X)), \neg \mathcal K_X)\).
\end{example}
\begin{proof}
    Let \(X_0\) be  the Fortissimo space on the reals (\href{https://topology.pi-base.org/spaces/S000022}{S22}).
    Recall that this topology adds a point \(\infty \not\in \mathbb R\) to \(\mathbb R_d\), the set of the real numbers with the discrete topology,
    so that the neighborhood basis at \(\infty\) consists of co-countable subsets of \(\mathbb R\).
    Finally, let \(X\) be the disjoint union, \([0,1] \sqcup X_0\), of the closed unit interval \([0,1]\) with \(X_0\).
    As the disjoint union of two \(T_5\) spaces (see \cite{Counterexamples}), \(X\) is \(T_5\).
    
    The winning strategy for P1 in the game \(\mathsf G_1(\mathcal N(\mathsf K(Y)), \neg \mathcal K_Y)\)
    where \(Y\) is the Fortissimo space on discrete \(\omega_1\) is described in
    \cite[Ex. 3.24]{CHCompactOpen}.
    We will verify that a modification to this idea works for \(X\).
    Throughout, we will use the identification \(A \mapsto \mathcal N_A\), \(\mathcal A \to \mathcal N(\mathcal A)\).
    
    P1 starts by playing \[\sigma(\emptyset) = K_0 = \{\infty\} \cup [0,1].\]
    Then, for \(n \in \omega\), suppose we have \(\langle K_j : j \leq n \rangle\)
    and \(\langle A_j : j < n \rangle\) defined where \(K_0 \subseteq K_j\) for each \(j \leq n\),
    and \(A_j = \{ x_{j,k} : k \in \omega \} \subseteq X \setminus K_j\) for each \(j < n\).
    P2 must respond to \(K_n\) with some open set that covers \(K_n\).
    Since \(\{\infty\} \cup [0,1] \subseteq K_n\), P2's move can be written as
    \(X \setminus A_n\) where \[A_n = \{ x_{n,j} : j \in \omega \} \subseteq X \setminus K_n \subseteq X_0.\]
    P1 responds to \(X \setminus A_n\) with
    \[\sigma(\langle X \setminus A_j : j \leq n \rangle) = K_{n+1} = K_0 \cup \{ x_{j,k} : j,k \leq n \}.\]
    This defines the strategy \(\sigma\) for P1.
    
    We now show that \(\sigma\) is winning.
    So consider any run of the game according to \(\sigma\), as coded above with
    \(\langle K_n : n \in \omega\rangle\) and \(\langle A_n : n \in \omega \rangle\)
    where \(A_n = \{ x_{n,k} : k \in \omega \} \subseteq X_0\).
    Let \(A = \bigcup \{ A_n : n \in \omega \}\) and note that \(A\) is countable.
    Let \(K\) be a compact subset of \(X\).
    Since \(X_0\) is anticompact (every compact subset of \(X_0\) is finite), there must be some \(n \in \omega\)
    for which \(K \cap A \subseteq K_n \subseteq U_n := X \setminus A_n\).
    
    We show that \(K \subseteq U_n\).
    We already have that \(K \cap A \subseteq U_n\).
    Also, observe that \(K \cap [0,1] \subseteq K_n \subseteq U_n\).
    Now, for any \(x \in ( K \cap X_0 ) \setminus A\),
    \[x \in X_0 \setminus \bigcup_{j\in\omega} A_j = \bigcap_{j\in\omega} X_0 \setminus A_j \subseteq X_0 \setminus A_n \subseteq U_n.\]
    Hence, \(\sigma\) is a winning strategy for P1 in \(\mathsf G_1(\mathcal N(\mathsf K(X)), \neg \mathcal K_X)\).
    
    We now argue that P1 cannot win \(\mathsf G_1(\mathcal N(\mathsf K(X)), \neg \mathcal K_X)\) with a pre-determined strategy.
    By the results of \cite{ClontzDualSelection}, \(\mathsf G_1(\mathcal N(\mathsf K(X)), \neg \mathcal K_X)\) is dual
    to \(\mathsf G_1(\mathcal K_X, \mathcal K_X)\); a particular application of that duality is that
    P1 has a pre-determined winning strategy in \(\mathsf G_1(\mathcal N(\mathsf K(X)), \neg \mathcal K_X)\)
    if and only if P2 has a winning Markov strategy in \(\mathsf G_1(\mathcal K_X, \mathcal K_X)\).
    Since \(X_0\) is anticompact and uncountable, \(X\) is not hemicompact.
    Hence, by Corollary \ref{cor:T1Hemicompact}, P2 does not have a winning Markov strategy in the \(k\)-Rothberger game on \(X\).
    Thus, P1 does not have a pre-determined winning strategy in \(\mathsf G_1(\mathcal N(\mathsf K(X)), \neg \mathcal K_X)\).
        
    Lastly, we argue that P1 doesn't have a winning strategy in the finite-open game on \(X\).
    We do this by showing that P2 actually has a winning Markov strategy in this game.
    Suppose P1 has played \(F_n \in [X]^{<\omega}\).
    Let \(V_n = X_0\) and \(W_n \subseteq [0,1]\) be an open set with \(F_n \cap [0,1] \subseteq W_n\)
    such that the Lebesgue measure of \(W_n\) is less than \(\dfrac{1}{2^{n+2}}\).
    Then let P2 respond to \(F_n\) with \(U_n := V_n \cup W_n\).
    Since the Lebesgue measure of \(\bigcup_{n\in\omega} W_n\) is less than \(1/2\),
    P2 has won the finite-open game.
\end{proof}

\begin{example}[\href{https://topology.pi-base.org/spaces/S000043}{S43}]
    The Sorgenfrey line is an example of a space which is Lindel\"{o}f but not
    \(\omega\)-Lindel\"{o}f. This follows from the fact that the Sorgenfrey plane is not Lindel\"{o}f;
    see \cite{Counterexamples} and Theorem \ref{thm:GerlitsNagy}.
\end{example}
\begin{example}[\href{https://topology.pi-base.org/spaces/S000025}{S25}] \label{example:Reals}
    The space of reals \(\mathbb R\) is an example of a space which is
    \begin{enumerate}[label=(\arabic*),ref=(\arabic*)]
        \item \label{firstReals}
        Markov \(\omega\)-Menger and not Markov \(\omega\)-Rothberger (see Theorems \ref{thm:MarkovOmega} and \ref{thm:OmegaSelectionChar}
        and the fact the reals are not even Rothberger, which is witnessed by any sequence of open covers consisting of intervals with exponentially decreasing diameters),
        \item Strategically \(\omega\)-Menger and not strategically \(\omega\)-Rothberger (same as in \ref{firstReals}),
        \item \(\omega\)-Menger and not \(\omega\)-Rothberger (same as in \ref{firstReals}),
        \item \label{fourthReals} Markov \(k\)-Rothberger and not Markov \(\omega\)-Rothberger (see Corollary \ref{cor:HemicompactMarkov},
        Theorem \ref{thm:OmegaSelectionChar} and the fact that the reals are not even Rothberger),
        \item Strategically \(k\)-Rothberger and not strategically \(\omega\)-Rothberger (same as in \ref{fourthReals}), and
        \item \(k\)-Rothberger and not \(\omega\)-Rothberger (same as in \ref{fourthReals}).
    \end{enumerate}
\end{example}
\begin{example}[\href{https://topology.pi-base.org/spaces/S000027}{S27}]
    The space of rationals \(\mathbb Q\) is an example of a space which is
    \begin{enumerate}
        \item \label{firstRationals} Markov \(\omega\)-Menger and not Markov \(k\)-Menger,
        \item Strategically \(\omega\)-Menger and not strategically \(k\)-Menger,
        \item \(\omega\)-Menger and not \(k\)-Menger,
        \item Markov \(\omega\)-Rothberger and not Markov \(k\)-Rothberger,
        \item Strategically \(\omega\)-Rothberger and not strategically \(k\)-Rothberger, and
        \item \(\omega\)-Rothberger and not \(k\)-Rothberger.
    \end{enumerate}
\end{example}
\begin{proof}
    See Theorem \ref{thm:MarkovOmega}, Proposition \ref{prop:T1FirstCountableHemicompact}, and the fact that \(\mathbb Q\) is not hemicompact \cite{Counterexamples}.
\end{proof}
The following is well-known, but we record a proof of it here for inclusion in the literature.
\begin{proposition}\label{IrrNotMenger}
    The space of irrationals \(\mathbb R \setminus \mathbb Q\) is not Menger.
\end{proposition}
\begin{proof}
    Using continued fraction expansions, the irrationals are homeomorphic to the Baire space \(\omega^\omega\).
    Let \(\mathscr U_n = \{[t] : t \in \omega^{<\omega}\}\), where \([t] = \{f \in \omega^\omega : f \mbox{ extends } t\}\).
    Suppose that \(F_n \subseteq \mathscr U_n\) are finite. We can find an \(f \in \omega^\omega\) so that \(f \notin \bigcup_n F_n\),
    which shows that the irrationals are not Menger. First pick \(f(0)\) so that \([\langle f(0) \rangle ] \notin F_0\).
    Then if \(f(0),\cdots,f(n)\) have been chosen so that \([\langle f(0),\cdots,f(k) \rangle ] \notin F_k\) for \(k \leq n\),
    we can choose \(f(n+1)\) so that \([\langle f(0),\cdots,f(n+1) \rangle ] \notin F_{n+1}\).
    Continue recursively in this way to produce \(f\) as desired.
\end{proof}
\begin{example}[\href{https://topology.pi-base.org/spaces/S000028}{S28}]
    The space of irrationals \(\mathbb R \setminus \mathbb Q\) is an example of a space which is
    \begin{enumerate}
        \item \(\omega\)-Lindel\"{o}f and not \(\omega\)-Menger (the irrationals are second-countable,
        so by Proposition \ref{prop:SecondCountable} are \(\omega\)-Lindel\"{o}f; from Proposition \ref{IrrNotMenger}
        the irrationals are not Menger,
        so by Theorem \ref{thm:OmegaSelectionChar}, cannot be \(\omega\)-Menger), and
        \item \(k\)-Lindel\"{o}f and not \(k\)-Menger (the irrationals are second-countable,
        so by Proposition \ref{prop:SecondCountable} are \(k\)-Lindel\"{o}f, and they are not Menger as before).
    \end{enumerate}
\end{example}
\begin{example}[\href{https://topology.pi-base.org/spaces/S000022}{S22}] \label{example:MainFortissimo}
    The Fortissimo space on the reals is an example of a space which is
    \begin{enumerate}
        \item Strategically \(k\)-Menger and not Markov \(k\)-Menger,
        \item Strategically \(k\)-Rothberger and not Markov \(k\)-Rothberger,
        \item Strategically \(\omega\)-Menger and not Markov \(\omega\)-Menger, and
        \item Strategically \(\omega\)-Rothberger and not Markov \(\omega\)-Rothberger.
    \end{enumerate}
\end{example}
\begin{proof}
    By similar reasoning to the argument in Example \ref{example:CHQ}, \(X\) is strategically \(k\)-Rothberger and strategically \(\omega\)-Rothberger.
    These in turn, imply that it is strategically \(k\)-Menger and strategically \(\omega\)-Menger.
    However, the Fortissimo space is not hemicompact and is not countable, so by Corollaries \ref{cor:BasicMarkovOmega}
    and \ref{cor:HemicompactMarkov} it is not Markov for any of those games.
\end{proof}

Let \(\mathbb K(X)\) be the set \(\mathsf K(X)\) endowed with the Vietoris topology (see \cite{MichaelSubsets}), as mentioned above.
Let \(\mathcal P_{\mathrm{fin}}(X)\) denote the set \([X]^{<\omega}\) viewed as a subspace
of \(\mathbb K(X)\).
There are many equivalences between properties studied in this paper and these hyperspaces
in \cite{CHVietoris}.
For example, a space \(X\) is \(\omega\)-Rothberger if and only if \(\mathcal P_{\mathrm{fin}}(X)\)
is Rothberger.
Analogous equivalences hold at the Lindel\"{o}f level and also for finite-selections.
Notably, the equivalence that fails is that of \(X\) being \(k\)-Rothberger and \(\mathbb K(X)\) being Rothberger,
as witnessed by \(\mathbb R\).

From results discussed in this work,
every \(k\)-Lindel\"{o}f space is \(\omega\)-Lindel\"{o}f.
As a consequence, if \(\mathbb K(X)\) is Lindel\"{o}f, then the dense subspace \(\mathcal P_{\mathrm{fin}}(X)\)
is also Lindel\"{o}f.
However,
\begin{example}
    There is a \(T_5\) space \(X\) such that \(\mathbb K(X)\) is strategically \(k\)-Rothberger
    and strategically \(\omega\)-Rothberger, but not hereditarily Lindel\"{o}f.
    Consequently, none of the Menger or Rothberger variants discussed in this paper, excluding the Markov
    levels, are necessarily hereditary for \(\mathbb K(X)\).
\end{example}
\begin{proof}
    Let \(X\) be the Fortissimo space on the reals
    and let \(\mathbb R_d\) be the reals with the discrete topology.
    As asserted in \cite{Counterexamples}, \(X\) is \(T_5\).
    Since \(X\) is anticompact, \(\mathbb K(X) = \mathcal P_{\mathrm{fin}}(X)\).
    As mentioned in Example \ref{example:MainFortissimo}, \(X\) is strategically \(k\)-Rothberger and
    strategically \(\omega\)-Rothberger.
    By \cite[Thm. 4.8]{CHVietoris}, \(\mathbb K(X) = \mathcal P_{\mathrm{fin}}(X)\) is strategically \(\omega\)-Rothberger;
    consequently, by anticompactness, \(\mathbb K(X)\) is also strategically \(k\)-Rothberger.
    Indeed, if \(\mathbf K \subseteq \mathbb K(X)\) is compact,
    \(\bigcup \mathbf K \subseteq X\) is compact (see \cite{MichaelSubsets}), hence finite.
    That means that \(\mathbf K\) is finite, as well.
    Lastly, \([\mathbb R_d]^{<\omega}\) is an uncountable relatively discrete subspace of \(\mathbb K(X)\),
    so \(\mathbb K(X)\) is not hereditarily Lindel\"{o}f; moreover, \(\mathbb K(X)\) is neither hereditarily
    Rothberger nor hereditarily Menger.
\end{proof}

Recall that a Luzin subset of \(\mathbb R\) is an uncountable set \(X\) such that
\(X \cap F\) is countable for every closed and nowhere dense \(F\).
Evidently, no Luzin set is meager.

Rothberger \cite{Rothberger} showed that every Luzin set is Rothberger.
Hence, if a Luzin set exists, it is a non-meager Rothberger subset of \(\mathbb R\).

As a response to a question of Galvin, Rec{\l}aw \cite{Reclaw} showed that every Luzin set is undetermined
in the point-open game.
By duality (see \cite{Galvin1978,ClontzDualSelection}), Luzin sets are undetermined for the Rothberger game, as well.
We now collect some facts about Luzin sets related to other properties investigated in this work.
\begin{lemma} \label{lem:LuzinNotStratMenger}
    No Luzin set is strategically Menger or \(k\)-Rothberger.
    Consequently, every Luzin set is an example of a Menger space which is not strategically Menger;
    an example of a Rothberger space which is not strategically Rothberger.
\end{lemma}
\begin{proof}
    We first argue that a Luzin set cannot be \(\sigma\)-compact.
    So let \(X \subseteq \mathbb R\) be a Luzin set and note that \(X\) cannot have
    a non-empty interior since any open subset of \(\mathbb R\) contains a copy of the Cantor set.
    Hence, every compact subset of \(X\) is nowhere dense.
    By definition, it follows that every compact subset of \(X\) must be countable.
    Hence, \(X\) is not \(\sigma\)-compact.

    As a subspace of \(\mathbb R\), \(X\) is second-countable and regular.
    Note that \cite[Lemma 4.10]{ClontzApplicationsOfLimitedInformation} asserts that
    any second-countable strategically Menger space is Markov Menger.
    By Corollary \ref{cor:BasicMarkovOmega},
    a regular Markov Menger space is \(\sigma\)-compact.
    Hence, since \(X\) is not \(\sigma\)-compact, it cannot be Markov Menger and hence, \(X\) cannot be strategically Menger.
    
    Now, since \(X\) is not \(\sigma\)-compact, it is certainly not hemicompact.
    As a subspace of \(\mathbb R\), \(X\) is \(T_1\) and first-countable, so \(X\) cannot be
    \(k\)-Rothberger by Proposition \ref{prop:T1FirstCountableHemicompact}.

    Lastly, by Rothberger's result \cite{Rothberger}, every Luzin set is Rothberger, hence Menger.
    So the final assertion holds.
\end{proof}
In fact, a bit more is true about Luzin sets.
As discussed in \cite{MillerSpecialSubsets}, every Luzin set is universally null, which means
that they have outer measure zero with respect to every continuous Borel probability measure on \(\mathbb R\).
Combining this with \cite[Thm. 3.9]{CKEBP}, no Luzin set can even have the Baire property.

In a separate line of investigation, Scheepers \cite{ScheepersLusin}
discusses selection principle properties of Luzin sets involving
families of open sets with the property that their closures cover the space in question.

\begin{example} \label{ex:LuzinNotOmega}
    Assuming CH, there is a Luzin set which is not \(\omega\)-Rothberger.    
    Indeed, \cite[Lemma 2.6 and Thm 2.8]{CombOpenCovers2} shows that there is a Luzin subset \(L\) of the reals
    such that \(L\) does not satisfy a selection principle \(\mathsf U_{\mathrm{fin}}(\Gamma,\Omega)\).
    According to \cite[Fig. 3]{CombOpenCovers2}, \(\mathsf S_1(\Omega,\Omega) \implies \mathsf U_{\mathrm{fin}}(\Gamma,\Omega)\).
    Hence, \(L\) is not \(\omega\)-Rothberger.
    Now, since every Luzin set is Rothberger, \(L\) is Rothberger without being \(\omega\)-Rothberger.
    Note, moreover, by Theorem \ref{thm:OmegaSelectionChar}\ref{OmegaSelectionRothbergerChar}, \(L\) has some finite power which fails to be Rothberger.
\end{example}
\begin{example}
    Assuming CH, \cite[Thm. 2.13]{CombOpenCovers2} proves the existence of a Luzin set which is \(\omega\)-Rothberger.
    Observe that, by Theorem \ref{thm:StrategicOmega} and Lemma \ref{lem:LuzinNotStratMenger},
    this Luzin set is not strategically \(\omega\)-Rothberger; by the same results,
    this Luzin set is \(\omega\)-Menger but not strategically \(\omega\)-Menger.
\end{example}

\begin{example} \label{ex:OmegaLindelofNotKLindelof}
    In \cite[Ex. 6.4]{CostantiniHolaVitolo}, under MA, a countable space is constructed
    which is not \(k\)-Lindel\"{o}f.
    Since every countable space is trivially \(\omega\)-Lindel\"{o}f, this is an example of a space
    which is \(\omega\)-Lindel\"{o}f but not \(k\)-Lindel\"{o}f.
\end{example}

\section{Open Questions} \label{section:Questions}

We finish with a list of questions.
\begin{question} \label{question:Undetermined}
    Are there spaces \(X\) for which the games \(\mathsf G_1(\mathcal K_X, \mathcal K_X)\) and
    \(\mathsf G_{\mathrm{fin}}(\mathcal K_X, \mathcal K_X)\) are undetermined (that is, that neither player has a winning strategy)?
\end{question}
This two-fold question may reduce to a single question if the single- and finite-selection versions turn out to be the same.
\begin{question}
    Does (strategically) \(k\)-Menger imply (strategically) \(k\)-Rothberger?
\end{question}
Note that, by Proposition \ref{prop:T1FirstCountableHemicompact} and Corollary \ref{cor:HemicompactMarkov},
no metrizable space can be an affirmative response to Question \ref{question:Undetermined}.

It is shown in \cite[Thm. 2.16]{FinitePowersMenger} that, consistently, there are two sets \(X\) and \(Y\) of reals which are \(\omega\)-Menger,
but that \(X \times Y\) is not even Menger.
That is, it is consistent with ZFC that the property of being \(\omega\)-Menger is not finitely productive.
\begin{question}
    Is it consistent with ZFC that any of the properties \(\omega\)-Lindel\"{o}f, \(\omega\)-Menger,
    or \(\omega\)-Rothberger are (finitely) productive?
\end{question}
\begin{question}
    Are any of the properties \(k\)-Lindel\"{o}f, \(k\)-Menger, or \(k\)-Rothberger (finitely) productive?
\end{question}

Corollary \ref{cor:HemicompactMarkov} asserts that the properties of being Markov \(k\)-Menger and
being Markov \(k\)-Rothberger are equivalent for regular spaces.
\begin{question}
    Are the properties of being Markov \(k\)-Menger and Markov \(k\)-Rothberger equivalent
    for all topological spaces?
    Alternatively, is there an example of a (preferably \(T_1\) or \(T_2\)) space which is Markov \(k\)-Menger but not Markov \(k\)-Rothberger
    (equivalently, not hemicompact, by Corollary \ref{cor:T1Hemicompact})?
\end{question}
As a related group of questions,
\begin{question}
    Is there an example of a space which is hemicompact, but not relatively hemicompact?
    Is there an example of a space which is relatively hemicompact, but not hemicompact?
    Is there a nearly hemicompact space which is not hemicompact?
\end{question}

As pointed out in Example \ref{ex:OmegaLindelofNotKLindelof}, \cite[Ex. 6.4]{CostantiniHolaVitolo}
offers an example of a space, assuming MA, that is \(\omega\)-Lindel\"{o}f but not
\(k\)-Lindel\"{o}f.
\begin{question} \label{question:OmegaLindelofNotKLindelof}
    Is there a ZFC example of an \(\omega\)-Lindel{\"{o}}f space which is not \(k\)-Lindel{\"{o}}f?
\end{question}

As Example \ref{ex:LuzinNotOmega} demonstrates, it is consistent with ZFC that there
exists a Rothberger subset of the reals which is not \(\omega\)-Rothberger.
As a contrast to this, it is well-known that, in Laver's model \cite{Laver}, every Rothberger subset of the reals is countable,
and hence \(\omega\)-Rothberger (see \cite{BorelConjecture_LaverMillerSacksModels} and \cite{Rothberger1941}).
Since these assertions are about subsets of reals, one can ask about more general settings.
\begin{question}
    Is it consistent with ZFC that every Rothberger space is \(\omega\)-Rothberger?
\end{question}

In \cite{Zdomskyy}, it is shown to be consistent with ZFC that,
for all metrizable spaces, the properties of being Menger and \(\omega\)-Menger
are equivalent.
Of course, every \(\omega\)-Menger space is Menger without any separation axiom assumptions.
So we are left with
\begin{question}
    Is it consistent with ZFC that every Menger space is \(\omega\)-Menger?
\end{question}

\providecommand{\bysame}{\leavevmode\hbox to3em{\hrulefill}\thinspace}
\providecommand{\MR}{\relax\ifhmode\unskip\space\fi MR }
\providecommand{\MRhref}[2]{%
  \href{http://www.ams.org/mathscinet-getitem?mr=#1}{#2}
}
\providecommand{\href}[2]{#2}

\end{document}